\documentclass[10pt]{article}

\usepackage{pstable}
\usepackage[colorlinks=true, unicode]{hyperref}

\begin{document}

\begin{title}
{On $p$-stablility in groups and fusion systems}
\end{title}
\author{L.~Héthelyi, M.~Szőke and A.~Zalesski}
\maketitle

\begin{abstract}
The aim of this paper is to generalise the notion of $p$-stability ($p$
is an odd prime) in finite group theory to fusion
systems. We first compare the different definitions of $p$-stability for
groups and examine properties of $p$-stability concerning subgroups and
factor groups. Motivated by Glauberman's theorem, we study the question how
$Qd(p)$ is involved in finite simple groups. We show that with a single
exception a simple group involving $Qd(p)$ has a subgroup isomorphic to
either $Qd(p)$ or a central extension of $Qd(p)$ by a cyclic group of order
$p$. Then we define $p$-stability for fusion systems and
characterise some of its properties. We prove a fusion theoretic version of
Thompson's maximal subgroup theorem. We introduce the notion of section
$p$-stability both for groups and fusion systems and prove a version of
Glauberman's
theorem to fusion systems. We also examine relationship between solubility
and $p$-stability for fusion systems and determine the simple groups whose
fusion systems are $Qd(p)$-free.
\end{abstract}

\section*{Introduction}
%Throught this paper, l
Throughout, let $p$ be an odd prime. The concept of $p$-stability
goes back to the middle of the $1960$s. It was originally defined by
D.~Gorenstein and J.~H.~Walter in \cite{gorenstein:walter:64} but, since
then, it has undergone several modifications. $p$-stability was investigated
by G.~Glauberman and also played a role in the classification of finite
simple groups. In the $1960$s, several different definitions of
$p$-stability arose and, at a first sight, these definitions appear not
to be equivalent. In Section~\ref{pstabsumm} of the present paper we go
around the notion of $p$-stability and examine some basic properties
that do not seem to have been considered so far. We show that $p$-stability
inherits to subgroups but not to factor groups. The smallest group which is
not $p$-stable is the semidirect product of $SL_2(p)$ with an elementary
Abelian group of order $p^2$ (acted on by $SL_2(p)$ in the natural way).
%2-dimensional representation).
%vector space over $\F_p$ with the natural action).
Glauberman denoted this group
by $Qd(p)$ and showed that a group does not involve $Qd(p)$ if and only
if all of its sections are $p$-stable. 
For further investigation, we define the concept of section $p$-stability
and give a new version of Glauberman's theorem (see
\ref{glaubpropforsections}).

%Several properties of groups can be considered `locally', that is, within
%the normalisers of non-trivial $p$-subgroups. Moreover, the operation of
%a group on its $p$-subgroups (by conjugation) had been extensively studied
%and led to the definition of a (saturated) fusion system, which can be
%considered as generalisation of the notion of a group. This concept was
%introduced by L.~Puig in the $1990$s and was originally called a `Frobenius
%category' (see~\cite{puig:06}). We give the exact definition of a fusion
%system in Section~\ref{prelims}. In the past 2 decades, fusion systems were
%studied thoroughly and many concepts of group theory such as solubility
%or simplicity were defined for fusion systems. Several group theoretical
%results has been proved to be valid also for fusion systems. Even
%$Qd(p)$-free fusion systems were defined and studied in
%\cite{kessar:linckelmann:2008}. In Section~\ref{secpstabfs} of the present
%paper we define $p$-stability for saturated fusion systems and investigate
%its basic properties. It turns out that, unlike finite groups,
%solubility does not imply $p$-stability (not even for $p > 5$).
%
Motivated by this result, we ask the question
%In the next three sections, we answer the question
which finite simple groups involve $Qd(p)$. Not only do we answer this
question, but we also investigate how $Qd(p)$ is involved. Our main result
can be summarised as follows:

\begin{thm}
Let $G$ be a finite simple group. Then $G$ involves $Qd(p)$ if and only
if $G$ is non-$p$-stable. This happens if and only if $G$ has a subgroup
isomorphic to $Qd(p)$ or a central extension of $Qd(p)$ by a cyclic group
of order $p$ or $G = He$ and $G$ contains an extension of $Qd(p)$ by a
Klein $4$-group.
\end{thm}

%We show that
%either $Qd(p)$ or a central extension of $Qd(p)$ by a cyclic group of order
%$p$ is a subgroup of a given simple group or $Qd(p)$ is not involved in it.
%The only exception is the Held group for $p = 3$, which contains a
%subgroup where a non-split extension of $SL_2(3)$ by a Klein $4$-group acts
%on an elementary Abelian group of order $9$. In any case, a simple group is
%$p$-stable if and only if it is section $p$-stable.
Our proof of Theorem~1 is divided into three parts: we examine the
alternating
groups and simple groups of Lie type in defining characteristic in
Section~\ref{Qdsecdefchar}. The sporadic simple groups are discussed in
Section~\ref{sporadic}. Finally, in Section~\ref{Qdsecnondefchar}, we
investigate simple groups of Lie type in non-defining characteristic.

Several properties of groups can be considered `locally', that is, within
the normalisers of non-trivial $p$-subgroups. Moreover, the operation of
a group on its $p$-subgroups (by conjugation) had been extensively studied
and led to the definition of a (saturated) fusion system, which can be
considered as generalisation of the notion of a group. This concept was
introduced by L.~Puig in the $1990$s and was originally called a `Frobenius
category' (see~\cite{puig:06}). We give the exact definition of a fusion
system in Section~\ref{prelims}. In the past 2 decades, fusion systems were
studied thoroughly and many concepts of group theory such as solubility
or simplicity were defined for fusion systems. Several group theoretical
results has been proved to be valid also for fusion systems. Even
$Qd(p)$-free fusion systems were defined and studied in
\cite{kessar:linckelmann:2008}. In Section~\ref{secpstabfs} of the present
paper we define $p$-stability for saturated fusion systems and investigate
its basic properties. It turns out that, unlike finite groups,
solubility does not imply $p$-stability (not even for $p > 5$).

In Section~\ref{secthompson}, we show a fusion theoretic version of
Thompson's maximal subgroup theorem (see~\cite[p. 295, Thm 8.6.3]
{gorenstein}). This can be summarised in the following way:

\begin{thm}
Let $\mf F$ be a saturated fusion system defined on the $p$-group $P$.
Let $\mf Q$ be a collection of subgroups of $P$ closed under $\mc
F$-morphisms. Let $\mf N$ be the set of normaliser systems of subgroups
of $P$ that are defined on elements of $\mf Q$. Assume each element of
$\mf N$ is constrained and $p$-stable. Then $\mf N$ has a unnique maximal
element.
\end{thm}

Then, in Section~\ref{secQdpfree}, we investigate $Qd(p)$-free fusion
systems and show the following:

\begin{thm}
A group does not involve $Qd(p)$ if and only if its fusion system is
$Qd(p)$-free.
\end{thm}

We define section $p$-stability for
fusion systems and prove a fusion theoretic version of Glauberman's
result (see Section~\ref{secsectpstab}):

\begin{thm}
A fusion system is section $p$-stable if and only if it is $Qd(p)$-free.
\end{thm}

As a consequence, we give a slight refinement of Glauberman's theorem, see
Theorem~\ref{refglaubthm}.

As the Sylow $p$-subgroups of $Qd(p)$ are extraspecial of exponent
$p$ and order $p^3$, we study the fusion systems defined on this group
in Secion~\ref{extraspecfussysts}. We show that with trivial exceptions
all of these fusion systems are non-$p$-stable and non-soluble.

Finally, we apply our group theoretic results to fusion systems and
investigate the relationship between solubility, $p$-stability and section
$p$-stability for fusion systems in Ssection~\ref{finalrems}.

%In the next three sections, we answer the question, which finite simple
%groups involve $Qd(p)$. But not only do we answer this question, but we
%also investigate how $Qd(p)$ is involved. We show that either $Qd(p)$
%or a central extension of $Qd(p)$ by a cyclic group of order $p$ is a
%subgroup of a given simple group or $Qd(p)$ is not involved in it. The
%only exception is the Held group for $p = 3$, which contains a subgroup,
%where a non-split extension of $SL_2(3)$ by a Klein $4$-group acts on an
%elementary Abelian group of order $9$. In any case, a simple group is
%$p$-stable if and only if it is section $p$-stable if and only if its
%fusion system is soluble. Our investigations are divided into three parts:
%we examine the alternating groups and simple groups of Lie type in defining
%characteristic in Section~\ref{Qdsecdefchar}. The sporadic simple groups are
%discussed in Section~\ref{sporadic}. Finally, in
%Section~\ref{Qdsecnondefchar}, we investigate simple groups of Lie type in
%non-defining characteristic.

\section{Summary on {\itshape p}-stable groups}\label{pstabsumm}
In the literature, we can find different definitions of $p$-stability
for groups. The notion of $p$-stability appears first
in~\cite[Def.~2, p. 171]{gorenstein:walter:64}, then in
\cite[p. 268]{gorenstein}. Later, Glauberman redefines this notion in
\cite[Def. 2.1 and 2.3, p. 1104]{glauberman:68} and
in~[p. 22]\cite{glauberman:71}.
 
Unfortunately, the four definitions are (pairwise) different and it is
not clear at all whether they are equivalent. For the sake of completeness,
we cite all four definitions. Glauberman proves that the definition
in~\cite{glauberman:71} is equivalent to that in~\cite{gorenstein}, but
the one in a later edition of the same book (see~\cite{gorenstein:07})
appears to be non-equivalent to that in~\cite{gorenstein}. Later in the
literature the definition in~\cite{glauberman:71} is used (see e. g.
in~\cite{Huppert:Blackburn} or \cite{solomon}). However, results
from~\cite{glauberman:68} have great importance and are oft cited, so the
equivalence of these definitions might be crucial. In the following, we
shall compare the two  definitions by examining some properties of
$p$-stability.

The original definiton of Gorenstein and Walter is the following:

\begin{Def}[Gorenstein--Walter, 1964]
Let $G$ be a finite group. Let $S$ be the largest soluble normal subgroup of
$G$. Let $p$ be a prime that divides $|S|$. Let $P$ be a Sylow $p$-subgroup of
$O_{p', p}(S)$ and $Q \leq P$ such that $(i)$ $O_{p'}(S)Q \triangleleft G$
and $(ii)$ $O_p\big(N_G(Q) / C_G(Q) P\big) = 1$. We shall say that $G$ is
{\itshape $p$-stable} provided the following condition holds for any such
subgroup $Q$:
\begin{itemize}
\item[] If $A$ is a $p$-subgroup that normalises $Q$ and satisfies the
commutator identity $[Q, A, A] = 1$, then $A \subseteq P C_G(Q)$.
\end{itemize}
\end{Def}

Gorenstein's advanced definition in \cite{gorenstein}:

\begin{Def}[Gorenstein, 1968]\label{gordef2}
Let $G$ be a finite group and $p$ an odd prime. $G$ is called {\itshape
$p$-stable} if the following condition is satisfied:
\begin{itemize}
\item[] If $K$ is a normal subgroup of $G$, $P$ is a $p$-subgroup of $K$
with $G = K N_G(P)$, and $A$ is a $p$-subgroup of $N_G(P)$ such that
$[P, A, A] = 1$, then
\[
A C_G(P) / C_G(P) \subseteq
O_p\big(N_G(P) / C_G(P)\big).
\]
\end{itemize}
\end{Def}

In~\cite{gorenstein:07}, the above group $K$ is specified as $O_{p'}(G)$.

The definition appearing in~\cite{glauberman:68} is as follows:

\begin{Def}[Glauberman, 1968]\label{pstabdef1}
Let $G$ be a finite group, let $p > 2$ be a prime, and let $\mc M(G)$ be the
set of subgroups $M$ of $G$ maximal with respect to the property that
$O_p(M) \ne 1$. $G$ is said to be {\itshape $p$-stable} if for all $M
\in \mc M(G)$ and for all $p$-subgroups $Q$ of $M$ such that $O_{p'}(M) Q
\triangleleft M$, whenever an element $x \in N_M(Q)$ has the property
that if
\[
[Q, x, x] = 1,
\]
then $x$ maps into $O_p\big(N_M(Q) / C_M(Q)\big)$ under the natural
homomorphism \break $N_M(Q)$ $\to N_M(Q) / C_M(Q)$.
\end{Def}

The revised definition of $p$-stability in~\cite{glauberman:71} is
the following:

\begin{Def}[Glauberman, 1971]\label{pstabdef2}
A group $G$ is said to be {\itshape $p$-stable} if for all $p$-subgroups
$Q$ of $G$ whenever an element $x \in N_G(Q)$ satisfies
\[
[Q, x, x] = 1,
\]
then $x$ maps into $O_p\big(N_G(Q) / C_G(Q)\big)$ under the natural
homomorphism \break $N_G(Q)$ $\to N_G(Q) / C_G(Q)$.
\end{Def}

\begin{rem}
\begin{enumerate}[$(i)$]
\item It can be easily checked that Gorenstein's subgroups $A$ can be
substituted by single elements $x$. Moreover, let $x = x_p x_{p'} \in
N_G(Q)$, where $x_p$ and $x_{p'}$ are commuting $p$- and $p'$-elements,
respectively. It is straightforward to check that if $[Q, x, x] = 1$,
then $x_{p'} \in C_G(Q)$. As a consequence, it can be assumed that $x$
is a $p$-element.

\item By any of the four definitions, every group with an Abelian Sylow
$p$-subgroup is trivially $p$-stable.

\item If we set $K = G$ in Definition~\ref{gordef2}, we obtain
Definition~\ref{pstabdef2}, so Gorenstein's definition implies
Glauberman's one.

\item It is less obvious, what the connection between the complicated
first definition and the other ones is. Since this definition was soon
revisited by Gorenstein himself, we shall not discuss this connection
here.
\end{enumerate}
\end{rem}

The smallest example for a group {\itshape not} being $p$-stable (by all
four definitions but we only check Glauberman's definitions) is the group
usually denoted by $Qd(p)$:

\begin{example}\label{qdp}
The group $Qd(p)$ is defined as a semidirect product of a two-dimensional
vector space $V$ over $\F_p$ with the special linear group $SL_2(p)$ via
the natural action:
\[
Qd(p) = V \rtimes SL_2(p).
\]
Clearly, $O_p(Qd(p)) = V \ne 1$, so $\mc M(G)$ consists solely of the
group itself. Since $O_{p'}(Qd(p)) = 1$, the subgroup $Q$ has to be
normal in $Qd(p)$. Hence $Q = V$ (or 1, but this case is trivial).
Now, $V$ is self-centralising, so $N_{Qd(p)}(V) / C_{Qd(p)}(V)
\cong
SL_2(p)$. The element
\[
x = \mkk 1 1 0 1 \in SL_2(p)
\]
satisfies the commutator relation $[Q, x, x] = 1$. Nevertheless,
$x$ is not contained in $O_p(SL_2(p))$ since the latter is trivial.
In the literature, this group is of great importance.
\end{example}

The next lemma gives a well-known description of $Qd(p)$ as a matrix group
(see Example~7.5 in \cite[p. 494]{HuppertII}):

\begin{lem}\label{QdpinPSL3}
$Qd(p)$ can be represented as a subgroup of $SL_3(p)$, namely, consisting
of  matrices of the form
\[
\mhh a b t c d u 0 0 1 ,
\]
where $ad - bc = 1$. This subgroup intersects
%(cosets with respect to 
$Z(SL_3(p))$ trivially and hence maps isomorphically into $PSL_3(q)$.
%of)
\end{lem}

As already mentioned, we shall focus on the latter two definitions of
Glauberman. The first question concerning $p$-stability is whether these
two definitions are equivalent. This question is important especially as
theorems proved with Definition~\ref{pstabdef1} in \cite{glauberman:68}
are often cited when using Definition~\ref{pstabdef2} of $p$-stability.
Nevertheless, this problem does not seem to have been dealt with.

A group $G$ with $O_p(G) \ne 1$ which is $p$-stable according to
Definition~\ref{pstabdef2} also satisfies Definition~\ref{pstabdef1},
simply because more subgroups $Q$ are considered there. There are also
some natural questions concerning $p$-stability which do not seem to have
been considered so far, such as whether a subgroup or a factor group of a
$p$-stable group is necessarily $p$-stable (according to any of the
definitions).

In the following, we answer the questions asked above.
In~\cite[p. 82]{gagen} it is shown that the semidirect product of
$A_8$ with an elementary Abelian group of order $3^8$ is $3$-stable
according to Definition~\ref{pstabdef1} and it contains a subgroup
isomorphic to $Qd(3)$. Hence this definition does not inherit to
subgroups. However, we can prove the following proposition using
Definition~\ref{pstabdef2} of $p$-stability:

\begin{prop}\label{def2tosubgroup}
Let $G$ be a group that is $p$-stable according to
Definition~$\ref{pstabdef2}$. Let $H$ be a subgroup of $G$. Then $H$ is
$p$-stable according to the same definition.
\end{prop}

\begin{proof}
Let $Q$ be a $p$-subgroup of $H$. Set $C = C_G(Q)$, $N = N_G(Q)$, $\bar N =
N / C$, $N_H = N_H(Q)$, $C_H = C_H(Q)$ and $\bar N_H = N_H / C_H$. As
$C_H = C \cap N_H$, we have
\[
\bar N_H \cong N_H C / C \leq \bar N,
\]
so the former can be naturally considered as a subgroup of the latter. Let
$x \in N_H$ such that $[Q,x,x]=1$. By Definition~\ref{pstabdef2}, $x C \in
O_p(\bar N) \cap \bar N_H \subseteq O_p(\bar N_H)$, whence the lemma. 
%Let $H$ be a subgroup of $G$ and $Q$ a $p$-subgroup of $H$. Assume the
%element $x \in N_H(Q)$ satisfies $[Q, x, x] = 1$. Then certainly
%$x \in N_G(Q)$, and hence $x C_G(Q) \in O_p\big(N_G(Q) / C_G(Q)\big)$
%as $G$ is $p$-stable. As $x \in N_H(Q)$ we have
%\begin{multline*}
%x C_G(Q) \in O_p\big(N_G(Q) / C_G(Q)\big) \cap N_H(Q) C_G(Q) / C_G(Q)
%\triangleleft \\
%\triangleleft N_H(Q) C_G(Q) / C_G(Q) \cong N_H(Q) / C_H(Q),
%\end{multline*}
%the latter by the second isomorphism theorem. Observe that
%$x C_G(Q)$ is mapped to $x C_H(Q)$ via the natural isomorphism
%\[
%N_H(Q) C_G(Q) / C_G(Q) \cong N_H(Q) / C_H(Q).
%\]
%Therefore, $x C_H(Q)$ is contained in a normal $p$-subgroup of $N_H(Q) /
%C_H(Q)$ and so is an element of $O_p\big(N_H(Q) / C_H(Q)\big)$, whence $H$
%is $p$-stable.
\qed\end{proof}

This proposition has three immediate consequences:

\begin{cor}
A group $G$ satisfying Definition~$\ref{pstabdef2}$ also satisfies
Definition~$\ref{pstabdef1}$.
\end{cor}

\begin{proof}
Assume $G$ is $p$-stable according to Definition~\ref{pstabdef2}.
Let $M \in \mc M(G)$ and let $Q \leq M$ with $QO_{p'}(M) \triangleleft M$.
By Proposition~\ref{def2tosubgroup} $M$ is $p$-stable by
Definition~\ref{pstabdef2}. Then for any $x \in N_M(Q)$ such that $[Q, x, x]
= 1$ we have $x C_M(Q) \in O_p\big(N_M(Q) / C_M(Q)\big)$, proving $G$ is
$p$-stable according to Definition~\ref{pstabdef1}.
\qed\end{proof}

\begin{cor}
Definition~$\ref{pstabdef1}$ does {\itshape not} imply
Definition~$\ref{pstabdef2}$, hence the two definitions are \emph{not}
equivalent.
\end{cor}

\begin{proof}
By~\cite[p. 82]{gagen}, the group $G = V \ltimes A_8$ is $3$-stable
according to Definition~\ref{pstabdef1}, but it is certainly not
$p$-stable according to Definition~\ref{pstabdef2} as $G$ contains a
subgroup isomorphic to $Qd(3)$ which is not $3$-stable.
\qed\end{proof}

\begin{cor}\label{localpstab}
A group $G$ is $p$-stable according to Definition~$\ref{pstabdef2}$ if and
only if $N_G(Q)$ is $p$-stable for all non-cyclic $p$-subgroups $Q$ of $G$.
\end{cor}

\begin{proof}
Note that $\Aut(Q)$ is Abelian if $Q$ is cyclic. So cyclic $p$-subgroups of
$G$ satisfy the $p$-stability condition, and hence this only needs to be
verified for non-Abelian subgroups Q.
%One direction is clear by Proposition~\ref{def2tosubgroup}. To prove the
%converse, observe first that if $Q$ is cyclic then $N_G(Q) / C_G(Q) \leq
%\Aut(Q)$ is Abelian, so every $p$-element is contained in $O_p\big(
%N_G(Q) / C_G(Q) \big)$.
%
%If $Q$ is non-cyclic, then $N_G(Q)$ is $p$-stable by assumption and hence
%for each $x$ in $N_G(Q)$
%\[
%[Q, x, x] = 1 \text{ implies } xC_G(Q) \in O_p\big(N_G(Q) / C_G(Q) \big),
%\]
%proving the $p$-stability of $G$.
\qed\end{proof}

From now on, we use Definition~\ref{pstabdef2} for $p$-stability (unless
otherwise stated explicitly).

The next question is about factor groups. In~\cite[p. 88]{gagen} it
is shown that $G/O_{p'}(G)$ is $p$-stable if $G$ is so. Although Gagen
uses Definition~\ref{pstabdef1}, the proof can be easily carried over to
Definition~\ref{pstabdef2}, too.

The next example shows that a factor group of a $p$-stable group need not
be $p$-stable in general. We are thankful to professor O.~Yakimova for
pointing out this example.

\begin{example}\label{pstabnonsectpstab_ex}
Let $p > 3$ and let $X$ and $Y$ be indeterminates over $\F_p$. Then the
polynomial ring $\F_p[X, Y]$ can be viewed as an $\F_p SL_2(p)$-module via
the action extending the natural operation on the $2$-dimensional vector
space $\langle X, Y \rangle_{\F_p}$. Let $W$ be the $p+1$-dimensional
subspace of $\F_p[X, Y]$ generated by the homogeneous polynomials of degree
$p$. Then the elements $X^p$, $X^{p-1}Y$, \ldots, $XY^{p-1}$, $Y^p$ form a
basis of $W$ and $W$ is an $\F_p SL_2(p)$-submodule. $W$ has
a single submodule $V = \langle X^p, Y^p \rangle_{\F_p}$. Note that
$SL_2(p)$ acts on $V$ via its natural representation. Consider the
group $G = W^* \rtimes SL_2(p)$, where $W^*$ denotes the module
contragredient to $W$. Since $W^*$ has a factor module isomorphic to $V^*
\cong V$, $G$ has a factor group isomorphic to $Qd(p)$. However, it can be
easily computed that the group $G$ itself is $p$-stable.
\end{example} 

In~\cite[Lemma~6.3.]{glauberman:68}, Glauberman proved a
characterisation of the groups all of whose sections are $p$-stable:

\begin{theorem}[Glauberman]\label{glaubthm}
Let $G$ be a finite group. Then the following two conditions are
equivalent:
\begin{enumerate}[$(i)$]
\item All sections of $G$ are $p$-stable;
\item $G$ does not involve $Qd(p)$.
\end{enumerate}
\end{theorem}

Theorem~\ref{glaubthm} implies that for $p \geq 5$ all $p$-soluble groups
are $p$-stable. The converse is obviously false: there are plenty of simple
groups whose Sylow $p$-subgroups are Abelian for some prime $p$.

Unfortunately, there is no nice characterisation of $p$-stable groups. It
is not true that a non-$p$-stable group necessarily has a subgroup
isomorphic to $Qd(p)$:

\begin{example}\label{tildeqdp}
The group $Qd(p)$ has a central extension with a cyclic group $Z$ of order
$p$: Let $E = \langle \tilde a, \tilde b \rangle $ be an
extraspecial group of exponent $p$. Denote its centre by $Z$ so that $Z =
\langle [\tilde a, \tilde b] \rangle$. Then $E / Z \cong V$ (the
normal subgroup of $Qd(p)$ of order $p^2$). Moreover, the images $a$ and
$b$ under the homomorphism $E \to V$ of $\tilde a$ and $\tilde b$,
respectively, generate $V$. It is well-known that the automorphism group of
$E$ has a subgroup isomorphic to $SL_2(p)$ and the action of
$SL_2(p)$ on  $\tilde a$ and $\tilde b$ is the same as on $a$ and $b$. Let
$\widetilde{Qd}(p) = E \rtimes SL_2(p)$ with the action just
defined. Then $\widetilde{Qd}(p)$ is non-$p$-stable as it is proven by the
subgroup $Q = E$ and $x \in SL_2(p)$ as in Example~\ref{qdp}. It is
easy to see that $\widetilde{Qd}(p)$ does not contain a subgroup isomorphic
to $Qd(p)$.
\end{example}

As we shall see later, $\widetilde{Qd}(p)$ has a representation as a subgroup
of $GL_p(q)$ if $p | q-1$ (see Lemma~\ref{uu2}.)
In order to give some more examples of non-3-stable groups, we now construct
$\widetilde{Qd}(3)$ as a subgroup of $GL_3(\mathbb C)$.

\begin{example}\label{tildeQd3}
Let $\rho$ be a (complex) primitive third root of unity. We define the
following complex matrices:
\[
\tilde a = \mhh {\rho} 0 0 0 {\rho^2} 0 0 0 1 , \hspace{2cm}
\tilde b = \mhh 0 1 0 0 0 1 1 0 0 , %\hspace{4mm}
\]
\[
x = \mhh 1 0 0 0 {\rho} 0 0 0 1 , \hspace{2cm}
t = \frac{1}{1 - \rho} \cdot \mhh 1 1 1 {\rho} {\rho^2} 1 {\rho^2} {\rho} 1 .
\]
A straightforward calculation shows that $E = \langle \tilde a$, $\tilde b
\rangle$ is an extraspecial group of order 27 and exponent 3, whereas, $S =
\langle x$, $t \rangle$ is isomorphic to $SL_2(3)$. Moreover, $S$
normalises $E$ and the operation of the elements $x$ and $t$ with respect
to the basis $a$, $b$ of $E / Z(E)$ is represented by the matrices $\smkk
1 0 1 1 $ and $\smkk 0 -1 1 0 $, respectively. Therefore, $\langle \tilde
a$, $\tilde b$, $x$, $t \rangle \cong \widetilde{Qd}(3)$.
\end{example}

The group in Example~\ref{tildeQd3} can be modified to obtain 2 more
non-3-stable groups of the same order:

\begin{example}\label{tildeQd3+}
We keep the notation of Example~\ref{tildeQd3}. Let $\theta$ be a primitive
ninth root of unity with $\theta^3 = \rho$ and let $x^- = \theta^{-1} x$ and
$x^+ = \theta x$. Define the groups $\widetilde{Qd}^-(3) = \langle a$, $b$,
$x^-$, $t \rangle$ and $\widetilde{Qd}^+(3) = \langle a$, $b$, $x^+$, $t
\rangle$. As the original group is `twisted' by a scalar matrix, all three
groups have the same image in $PSL_3(\mathbb C)$ (namely, a subgroup isomorphic to
$Qd(3)$. Hence all these groups are central extensions of $Qd(3)$ by a
cyclic group of order $3$. Moreover, the elements $x^+$ and $x^-$ together
with the subgroup $E$ show that $\widetilde{Qd}^+(3)$ and
$\widetilde{Qd}^-(3)$ are non-3-stable.
\end{example}

\begin{rem}\label{tildeqd3rem}
By construction, the group $\widetilde{Qd}^-(3)$ is contained in $SL_3(\mbb C)$
unlike the other two groups. An easy calculation shows that the centraliser
of a Sylow $2$-subgroup of $\widetilde{Qd}(3)$ (a subgroup of order 72)
contains an elementary Abelian group of order 9, while that in any of the
other two groups contains a cyclic group of order 9.

Further investigation shows that $\widetilde{Qd}^-(3)$ and
$\widetilde{Qd}^+(3)$ have non-isomorphic Sylow 3-subgroups.

Moreover, the Sylow 3-subgroups of all three groups have exponent 9 and the
those of $\widetilde{Qd}(3)$ and $\widetilde{Qd}^+(3)$ cannot be embedded
into $(C_9 \times C_9) \rtimes C_3$, the largest subgroup of $SL_3(\mbb C)$
of exponent 9.
%Moreover, $SL_3(\C)$ does not contain a subgroup isomorphic to a Sylow
%$3$-subgroup of $\widetilde{Qd}^+(3)$ and hence $\widetilde{Qd}^+(3)$
%cannot be embedded into $SL_3(\C)$.

Let $q = \ell^s$ such that $3|q-1$. Then reduction modulo $\ell$ carries
over the construction in Example~\ref{tildeQd3} to $GL_3(q)$. To see this
observe that $\F_q$ contains primitive third roots of unity in this case.

If, moreover, $9|q-1$, then $\F_q$ contains primitive ninth roots of unity
as well, and hence the constructions of Example~\ref{tildeQd3+} are valid
in $SL_3(q)$ and $GL_3(q)$.

Note that the above defined groups are minimal non-3-stable subject to
containment. The question naturally arises: which groups are minimal
non-$p$-stable? We do not answer this question in this paper, but in
section~\ref{sporadic}, we shall see one more example for the prime $p = 3$.
\end{rem}

%For the prime 3, there are even more examples of non-$p$-stable groups having
%no subgroups isomorphic to $Qd(p)$:
%
%\begin{example}
%Let $G$ be the group of size $648$ with number $531$ in the small-group
%library of {\sf GAP}. Then $G$ has a factor group isomorphic to $Qd(p)$
%by a normal subgroup of order $3$ contained in the centre of
%the Sylow 3-subgroups of $G$. By construction, $G$ is not $3$-stable
%as the preimage in $G$ (under the natural homomorphism) of the element $x$
%and the subgroup $V$ of Example~\ref{qdp} contradict the definition of
%$p$-stability. Furthermore, $G$ has no subgroups isomorphic to $Qd(3)$
%(actually, no one isomorphic to $SL_2(3)$).
%\end{example}

Although Theorem~\ref{glaubthm} was proved with Definition~\ref{pstabdef1}
of $p$-stability, the result is often used with Definition~\ref{pstabdef2}.
In fact, the theorem is cited in~\cite{glauberman:71}, where
Definition~\ref{pstabdef2} appears, without mentioning that the proof was
worked out with another definition. However, the next result is clear by
the above:

\begin{prop}
For a group $G$, the following are equivalent:
\begin{enumerate}[(i)]
\item All sections of $G$ are $p$-stable according to
Definition~$\ref{pstabdef1}$.
\item All sections of $G$ are $p$-stable according to
Definition~$\ref{pstabdef2}$.
\end{enumerate}
\end{prop}

For the proof observe that if $G$ has a non-$p$-stable section $H/K$
according to Definition~\ref{pstabdef2} proved by the subgroup $Q \leq
H/K$ and the element $x \in N_{H/K}(Q)$, then the section $N_{H/K}(Q)$
of $G$ is non-$p$-stable according to Definition~\ref{pstabdef1} (proved
by the same $p$-subgroup $Q$ and element $x$).

\medskip
After introducing some notation, we define a more general notion.
For $p$-subgroups $Q$, $R$ of $G$ such that $R \triangleleft Q$, we let
$N_G(Q/R)$ be the largest subgroup of $G$ that acts by conjugation on $Q/R$
and $C_G(Q/R)$ be the largest subgroup of $N_G(Q/R)$ that acts trivially
on $Q/R$. Note that
\[
N_G(Q/R) = N_G(Q) \cap N_G(R)
\]
and
\[
C_G(Q/R) = \{ x \in N_G(Q/R)\ |\ [Q, x] \subseteq R \}.
\]

\begin{Def}\label{secpstab}
A group $G$ is said to be {\itshape section $p$-stable} if for all
$p$-sub\-groups $R$ and $Q$ of $G$ such that $R \triangleleft Q$,
whenever an element $x \in  N_G(Q/R)$ satisfies $[Q, x, x] \subseteq R$,
then $x C_G(Q/R)$ is contained in $O_p\big(N_G(Q/R) / C_G(Q/R)\big)$.
\end{Def}

Clearly, any section $p$-stable group is $p$-stable.

\begin{prop}\label{glaubpropforsections}
For a group $G$, the following are equivalent:
\begin{enumerate}[$(i)$]
\item $G$ is section $p$-stable.
\item All sections of $G$ are $p$-stable.
\item $N_G(R) / R$ is $p$-stable for all $p$-subgroups $R$ of $G$.
\end{enumerate}
\end{prop}

\begin{proof}
The equivalence of $(i)$ and $(iii)$ is clear by the isomorphism theorems. Also, the implication $(ii) \to (iii)$ is trivial.

$(i) \Rightarrow (ii)$: Assume first that $G$ is section $p$-stable and let
$H/K$ be a section of $G$.
% By definition, $G$ is also $p$-stable and hence $H$ is $p$-stable by
% Lemma~\ref{def2tosubgroup}.
Let $T$ be a $p$-subgroup of $H/K$. Denote by
$Q$ a Sylow $p$-subgroup of the preimage of $T$ under the natural
homomorphism $H \to H/K$. Let $R = Q \cap K$. Then $T = KQ / K \cong  Q/R$.
Assume an element $\bar x \in N_{H/K}(T)$ satisfies $[T, \bar x, \bar x]
= 1$.

Let $x \in H$ be such that $xK = \bar x$. Observe that $Q^x \subseteq KQ$
as $T$ is normalised by $\bar x$. Since $Q$ is a Sylow $p$-subgroup of $KQ$,
we have $Q^x = Q^k$ for some $k \in K$. Hence $xk^{-1} \in N_H(Q)$ is also
a preimage of $\bar x$, so we may assume $x \in N_H(Q)$.

By assumption, $[Q, x, x] \subseteq K$, so $[Q, x, x] \subseteq Q \cap K = R$
as $Q$ is normalised by $x$. Now, as $G$ is section $p$-stable,
\[
x C_G(Q/R) \in O_p\big(N_G(Q/R) / C_G(Q/R)\big) \cap
\big( N_H(Q/R) \cdot C_G(Q/R) / C_G(Q/R) \big)
\]
follows. Since
\[
N_H(Q/R) / C_H(Q/R) \cong N_H(Q/R) \cdot C_G(Q/R) / C_G(Q/R),
\]
the coset $x C_H(Q/R)$ is contained in a normal $p$-subgroup of the factor
group $N_H(Q/R) / C_H(Q/R)$. The claim now follows because
\[
N_H(Q/R) / C_H(Q/R) \cong N_{H/K}(T) / C_{H/K}(T).
\]
(Observe that $N_H(KQ/K) = K\cdot N_H(Q/R)$ and $C_H(KQ/R) = K \cdot
C_H(Q/R)$ hold by straightforward calculations.)
%
%$(ii) \Rightarrow (iii)$ is trivial.
%
%$(iii) \Rightarrow (i)$:
%Assume now that each $N_G(R) / R$ is $p$-stable and let $Q$ be a
%$p$-subgroup of $N_G(R)$ containing $R$. (Then $R \triangleleft Q$.)
%Let $H = N_G(Q/R)$ so that $H/R$ is $p$-stable. Let $x \in
%N_G(Q/R)$ such that $[Q, x, x] \subseteq R$. Then $\bar x = x R$ satisfies
%$[Q/R, \bar x, \bar x] = 1$. Hence $\bar x C_{H/R}(Q/R) \in O_p\big((H/R)
%/ C_{H/R}(Q/R)\big)$  as $H/R$ is $p$-stable. The claim now follows as
%\[
%(H/R) / C_{H/R}(Q/R) \cong H / C_H(Q/R).
%\]
\qed\end{proof}

By Theorem~\ref{glaubthm}, a group is section $p$-stable if and only if
it does not involve $Qd(p)$.

%%%%%%%%%%%%%%%%%%%%%%%%%%%%%%%%%%%%%%%%%%%%%%%%%%%%%%%%%%%%%%%%%%%%%%%%%%%
\section{{\itshape Qd}({\itshape p}) as a section of simple groups}
\label{Qdsecdefchar}
We now discuss the problem which simple groups involve $Qd(p)$.
% By Lemma~\ref{Qdpfreesol}, if a group does not involve $Qd(p)$, then its
% fusion system is soluble. Thus the groups with non-soluble $p$-fusion
% systems do involve $Qd(p)$ and we only have to go through the list in 
% Theorem~\ref{flores-foote}.
More specifically, we want to examine how the group $Qd(p)$ is involved
in finite simple groups. This question is discussed in the next few
sections. Besides this, we also determine whether the simple group in
question is $p$-stable.
 
%%%%%%%%%%%%%%%%%%%%%%%%%%%%%%%%%%%%%%%%%%%%%%%%%%%%%%%%%%%%%%%%%%%%%%%%%%%
% \section{{\itshape Qd}({\itshape p}) as a section of simple groups}

% In the next few sections we examine the question how the group $Qd(p)$ is
% involved in simple groups. % of Lie type. 
This section is devoted to alternating groups and simple groups of Lie type
in defining characteristic.

\begin{theorem}\label{QdpinAn}
The alternating group $A_n$ has a subgroup which is isomorphic to $Qd(p)$ if and
only if $n \geq p^2$. For $n < p^2$, $Qd(p)$ is not involved in $A_n$.
Therefore, $A_n$ is $p$-stable for $n < p^2$ and non-p-stable otherwise.
\end{theorem}

\begin{proof}
As the Sylow $p$-subgroups of $A_n$ are Abelian if $n < p^2$, $Qd(p)$
cannot be involved in $A_n$ in this case.

$SL_2(p)$ has index $p^2$ in $Qd(p)$. The permutation representation of
$Qd(p)$ on the (right) cosets of $SL_2(p)$ is faithful as $Qd(p)$ has no
normal subgroup contained in $SL_2(p)$ rather than the trivial one. This
permutation representation gives an embedding of $Qd(p)$ into $A_{p^2}$
(observe that $Qd(p)$ has no subgroup of index 2) and hence into each
$A_n$ with $n \geq p^2$.

The statement on $p$-stability follows from the above.
\qed\end{proof}

The description of $Qd(p)$ as in Lemma~\ref{QdpinPSL3} gives the main part
of the following theorem:

\begin{theorem}\label{defchar}
Let $G$ be a simple group of Lie type of characteristic $p$. Then $Qd(p)$
is not involved in $G$ if and only if $G$ is of type $A_1$, ${^2A_2}$ or
${^2G_2(3^{2n+1})}$. If $G$ is of type $B_2$ or ${^2A_n}$ with $n \geq 3$,
then $G$ has a subgroup isomorphic to $\widetilde{Qd}(p)$. In all other cases, $G$ has a subgroup isomorphic to $Qd(p)$. Consequently, $G$ is
$p$-stable if and only if it does not involve $Qd(p)$.
%if and only if the type of $G$ is one of the following list:
%\begin{itemize}
%\item $A_n$ with $n \geq 2$,
%\item $B_n$ ($n \geq 2$) or $C_n$ with $n \geq 3$,
%\item $D_n$ with $n \geq 4$,
%\item $E_6$, $E_7$, $E_8$ or $F_4$.
%\end{itemize}
\end{theorem}

\begin{proof}
Note that the cases of $^2B_2$ and $^2F_4$ are irrelevant because they are
defined in characteristic 2.

The Ree groups $^2G_2(3^{2n+1})$ have Abelian Sylow $2$-subgroups, hence
they cannot involve $Qd(p)$. The simple groups of type $A_1$ have Abelian
Sylow $p$-sub\-groups, so they do not involve $Qd(p)$.

% The cases of $^2A_2$ (unitary groups) and
For the unitary groups $G = U_3(q^2)$, we can use the description of a
Sylow $p$-subgroup $P$ of $G$ as in~\cite[Satz 10.12, p. 242]{HuppertI}.
% A Sylow $p$-subgroup $P$ of
% $G$ consists of transformations represented by matrices of the form
% \[
% x(a, b) = \rhh 1 a b . 1 -a^q . . 1 ,
% \]
% where $Tr(b) + N(a) = 0$ (with trace and norm of $\F_{q^2} / \F_q$).
% but lengthy 
% $x(a, b)^g$ of $x(a, b) \ne 1$ 
A straightforward calculation shows the following: If a
conjugate of an element (different from $1$) of $P$ is contained in $P$,
then the conjugating element lies in the normaliser $N_G(P)$. Now,
$N_G(P)$ is the semidirect product of $P$ with a cyclic group $C$
of order dividing $q^2 - 1$. Since such a semidirect product cannot
involve $Qd(p)$, we have that no $p$-local subgroup of $G$
involves $Qd(p)$ and hence $G$ does not involve it, either.

Let $G=Sp_4(q)$ and let $X \cong Sp_4(p)$ be a subgroup of $G$. It is
well-known that the stabiliser in $X$ of a non-zero vector of the natural
$\F_p Sp_4(p)$-module is isomorphic to $\widetilde{Qd}(p)$.
%It is well known
%and can be easily checked that $O_p(\widetilde{Qd}(p))$ is an extraspecial
%group $E$ of order $p^3$ and exponent $ p$, and $\widetilde{Qd}(p) = E
%\cdot R$ is a semidirect product, where $R = Sp_2(p) \cong SL_2(p)$. In
%fact, $Z(\widetilde {Qd}(p))=Z(E)$ and $\widetilde {Qd}(p)/Z(E) \cong
%Qd(p)$. 
As $|Z(G)|=2$, $PSp_4(q)$ has a subgroup isomorphic to $\widetilde{Qd}(p)$.
 
Note that $SO_5(q)$ is isomorphic to $PSp_4(q)$.

For $n \geq 4$, the special unitary group $SU_n(q)$ contains a subgroup
isomorphic to $Sp_4(q)$ and hence it has a subgroup isomorphic to
$\widetilde{Qd}(p)$. Since $Z(SU_n(q))$ is a $p'$-group, the same is true
for $PSU_n(q)$.

All the other simple groups of Lie type ($A_n$ for $n \geq 2$,  $B_n$, 
$C_n$ for $n \geq 3$, $D_n$ and $^2D_n$ for $n \geq 4$, $E_n$ for $6 \leq n
\leq 8$, $F_4$, $G_2$, $^2E_6$, and $^3D_4$) are known to have a subgroup
isomorphic to $PSL_3(p)$, that is, $A_2(p)$ (for the exceptional groups,
see also \cite{liebeck:saxl:seitz}). Thus they all have subgroups
isomorphic to $Qd(p)$ by Lemma~\ref{QdpinPSL3}.
\end{proof}

%%%%%%%%%%%%%%%%%%%%%%%%%%%%%%%%%%%%%%%%%%%%%%%%%%%%%%%%%%%%%%%%%%%%%%%%%%%
\section{The case of simple groups of Lie type in non-defining
characteristic} \label{Qdsecnondefchar}

In this section, we discuss the question how $Qd(p)$ is involved in simple
groups of Lie type in non-defining characteristic. More precisely, $G$ is a
simple group of Lie type defined over the field $\F_q$, where $q$ is
a power of a prime $\ell \ne p$. This means $p$ differs from the defining
characteristic $\ell$ of $G$.

The main result of this section is the following:

\begin{theorem}\label{mainLie}
Let $G$ be a simple group of Lie type of characteristic $\ell \ne p$.
Suppose that the Sylow $p$-subgroups of $G$ are non-Abelian. Then one of
the following holds:
\begin{enumerate}[$(i)$]
\item $G$ contains a subgroup isomorphic to $\widetilde{Qd}(p)$;

\item Either $G \cong PSL_p(q)$ (whith $p | q-1$) or $G\cong PSU_p(q)$
(with $p | q+1$) or $p = 3$, $G \cong {}^3D_4(q)$, $F_4(q)$, ${}^2F_4(q)$,
(with $q=2^{2m+1}$ $m > 0$), or ${^2F}_4(2)'$ and $G$ contains a subgroup
isomorphic to $Qd(p)$;

\item $p = 3$, $9 | q^2 - 1$, $G = G_2(q)$ and $G$ contains a subgroup
isomorphic to $\widetilde{Qd}^-(3)$.

\item $p=3$ and $q^2-1$ is not a multiple of $9$, $G=G_2(q)$ and $G$ has no
section isomorphic to $Qd(3)$.

\end{enumerate}
Consequently, $G$ is $p$-stable if and only if it is section $p$-stable.
\end{theorem}

The conditions on a prime $p$ which guarantee that a Sylow $p$-subgroup of
a simple group $G$ is Abelian must be known to experts, but we have not
found any reference. So we write down these in Proposition \ref{p11}
for cases relevant to Theorem \ref{mainLie}, that is, for the cases where
$G$ is a simple group of Lie type defined over the field $\F_q$, $q =
\ell^s$ and $\ell \ne p$. Denote by $e_p(q)$ the order of $q$ modulo $p$,
that is, the smallest natural number $i$ such that $p |q^i-1$.

\begin{prop}\label{p11}
Let $G$ be a simple group of Lie type in characteristic $\ell \ne p$.

\begin{enumerate}[$(1)$]
\item Suppose that $p=3$ and the Sylow $3$-subgroups of $G$ are Abelian. Then one of
the following holds:
\begin{enumerate}[$(i)$]
\item $G \cong PSL_2(q)$, where $q > 2$;
\item $G \cong PSL_3(q)$, where $q-1\equiv 3$ or $6 \pmod 9$;
\item $G \cong PSL_n(q)$, where $3 | q+1$ and $2 < n < 6$;
\item $G \cong PSU_3(q)$, where $q > 2$ and $q+1\equiv 3$ or $6\pmod 9$;
\item $G \cong PSU_n(q)$, where $3 | q-1$ and $2 < n < 6$;
\item $G \cong B_2(q)$;
\item $G \cong {}^2B_2(q)$, where $q=2^{2m+1}$ and $m > 0$;

%\item $G \cong {}^2F_4(2)'$.
\end{enumerate}

\item Suppose that $p>3$ and the Sylow $p$-subgroups of $G$ are Abelian.
Then one of the following holds:
\begin{enumerate}[$(i)$]
\item $G \cong {}^2B_2(q)$, where $q = 2^{2m+1}$, $m > 0$;
\item $G \cong G_2(q)$;
\item $G \cong {}^2G_2(q)$, where $q = 3^{2m+1}$, $m > 0$;
\item $G \cong {}^2F_4(2)'$ or ${}^2F_4(q)$, where $q = 2^{2m+1}$, $m > 0$;
\item $G \cong {}^3D_4(q)$;
\item $G \cong F_4(q)$;
\item $G \cong E_6(q)$, where $p > 5$ or $p = 5 \ndiv q-1$;
\item $G \cong {}^2E_6(q)$, where $p > 5$ or $p = 5 \ndiv q+1$;  
\item $G \cong E_7 (q)$, where $p>7$ or $p = 5$ or 7 and $p \ndiv q^2-1$;
\item $G \cong E_8(q)$, where $p > 7$ or $p = 7 \ndiv q^2 - 1$ or $p = 5$;
\item $G \cong PSL_n(q)$, where $n<e_p(q)p$;

\item $G \cong PSU_n(q)$, where $2<n<2e_p(q)p$ if $e_p(q)$ is odd,
$2<n<e_p(q)p $ if $e_p(q) \equiv 0 \pmod 4$ and $n<e_p(q)p/2$ if $e_p(q)
\equiv 2 \pmod 4$;

\item $G \cong B_{n}(q)$, where $q$ is odd and $1<n<e_p(q)p $ if $e_p(q)$
is odd, $1<n<e_p(q)p/2$ if $e_p(q)$ is even;
\item $G \cong C_n(q)$, where $2<n<e_p(q)p $ if $e_p(q)$ is odd,
$2<n<e_p(q)p/2$ if $e_p(q)$  is even;

\item $G \cong D_{n}(q)$, where $3<n<e_p(q)p$ if $e_p(q)$ is odd and
$4<n\leq e_p(q)p/2$ if $e_p(q)$ is even;
  
\item $G \cong {}^2D_{n}(q)$, where $3<n\leq e_p(q)p$ if $e_p(q)$ is odd and
$4<n<e_p(q)p/2$ if $e_p(q)$ is even.
\end{enumerate}
\end{enumerate}
\end{prop}

We reach the proof of the above two results towards a series of lemmas which
we state and prove below.

\begin{lem}\label{hu2}
Let $m$, $n$ be positive integers, and let $c = \gcd(m, n)$, the greatest
common divisor of $m$ and $n$. Then $q^c-1$ is the greatest common divisor
of $q^m-1,q^n-1$. Furthermore,
 %if %$q$ is a prime power? and $e=e_p(q)$then 
$p$ divides $q^n-1$ if and only if $e_p(q)$ divides $n$.
 \end{lem}

\begin{proof}
The first statement is Hilfsatz~2(a) in \cite{huppert:1970}.
%By definition of
%$e$,  $p$ divides $q^e-1$ and does not divide $q^i-1$ for $i<e$. So if $p$
%divides $q^n-1$ then $n$ is a multiple of $e$.  Conversely, if $n=ke$ for
%some integer   $k>0$ then $q^e-1$ divides $q^{ke}-1$, whence the result. 
The second is an elementary consequence of that $e_p(q)$ is the order of
$q$ in the multiplicative group $\F^{*}_p$ of the field of $p$ elements.
\qed\end{proof}

%The proof of the statements of the following lemma can be found on pages
%532, 531, and 532 in \cite{Weir:1955}, respectively.
%
%\begin{lem}\label{ww3}
%\begin{enumerate}[$(i)$]
%\item If $G = U_n(q)$ and $e_p(q) \equiv 2\pmod 4$, then $G$ contains a
%Sylow $p$-subgroup of $GL_n(q^2)$.
%
%\item If $G = Sp_{2n}(q)$, and $e_p(q)$ is even, then $G$ contains a Sylow
%$p$ subgroup of $GL_{2n}(q)$.
%
%\item If $G = O_{2n+1}(q)$ and $e_p(q)$ is even, then $G$ contains a Sylow
%$p$-subgroup of $GL_{2n+1}(q)$.
%\end{enumerate}
%\end{lem}

%Now we discuss the case of linear and unitary groups.
\paragraph{Linear and unitary groups}

\begin{lem}\label{uu1}
Let $E$ be an extraspecial group of order $p^3$. If $p$
divides $q-1$ (resp. $q+1$), then $E$ is isomorphic to a subgroup of
$GL_p(q)$ (resp., $U_p(q)$).
\end{lem}

\begin{proof}
The statement on $GL_p(q)$ is well known. Let $ p$ divide $q+1$. Then $E$
is isomorphic to a subgroup of $GL_p(q^2)$. As $p>2$, a Sylow $p$-subgroup
of $U_p(q)$ is a Sylow $p$-subgroup in $GL_p(q^2)$, see
\cite[p. 532]{Weir:1955}, whence the statement.
\qed\end{proof}

\begin{lem}\label{ug1}
Let $G=U_n(q)$, and $P$ a Sylow $p$-subgroup of G. If $e_p(q) \equiv 2
\pmod 4$, then $P$ is a Sylow $p$-subgroup of $GL_n(q^2)$,
otherwise $P$ is isomorphic to a Sylow $p$-subgroup of $GL_l(q^2)$, where
$l$ is the integral part of $n/2$.
\end{lem}

\begin{proof}
If $e = e_p(q) \equiv 2 \pmod 4$, then this is in stated in
\cite[p. 532]{Weir:1955}. 
So we may assume that either $4|e$ or $e$ is odd.

Note that $U_n(q)$ contains a subgroup $X$ isomorphic to $GL_{l}(q^2)$. 
It suffices to prove the result for  $n=2l+1$. As
\begin{align*}
|U_n(q)| &= (q+1)(q^2-1)\cdots (q^n+1)q^a
\intertext{and}
|GL_l(q^2)| &= (q^2-1)\cdots (q^{2l}-1)q^b
\end{align*}
for some integers $b>a>1$, the index of $GL_l(q^2)$ in $U_n(q)$ equals
\[
(q+1)(q^3+1)\cdots (q^n+1)q^{a-b}.
\]
We show that this number is coprime to $p$. For this it suffices to observe
that $q^i+1$ is coprime to $p$ for $i$ odd. Suppose the contrary that 
$p|q^i+1$ for some $i$. Then $p|q^{2i}-1$. By Lemma \ref{hu2}, $e|2i$.

Let first $e$ be odd. Then $e|i$ and hence $p|q^i-1$, so $p\ndiv
q^i+1$.  
%% \itf $q^i+1$ is coprime to $p$, unless possibly $i$ is a multiple of $e$ . 
%Let $i=te$ for some integer $t\geq 1$. 
%Then $q^{te}+1=q^{te}-1+2$. As $p$ is odd and  $p|q^e-1$,
%it follows that $p$ does not divide $q^{te}+1$. 

Now let $e=2m$, where $m$ is even.  Then
% $p|q^m+1$. Moreover, 
$m|i$ as $e|2i$. This is a contradiction as  $m$ is even, whereas, $i$ is
odd.
\qed\end{proof}

\begin{lem}\label{uu2}
%Assume $p > 3$.
Let $G=GL_p(q)$ (resp.,  $U_p(q)$), so that $G'=SL_p(q)$
(resp., $SU_p(q)$). Suppose $p|q-1$  (resp. $p|q+1$). Then $G$ contains a
subgroup isomorphic to $\widetilde{Qd}(p)$. If  $p > 3$, then this subgroup
is contained in $G'$. Consequently, $Qd(p)$ is isomorphic to a subgroup of
$PGL_p(q)$ (resp. $PGU_p(q)$ and is contained in $PSL_p(q)$ (resp.,
$PSU_p(q)$) if $p > 3$.
\end{lem}

\begin{proof}
Set $Z=Z(G)$. Let $E$ be the  extraspecial group of order $p^3$ and
exponent $p$.  By Lemma~\ref{uu1}, there is a faithful representation
$\phi$: $E \to G$. Then the character $\chi$ of $\phi$ vanishes on $E
\setminus Z(E)$ \cite[9.20]{doerk:hawkes}. Then $(\chi,\chi)=1$, and hence
$\phi$ is
absolutely irreducible.   (As $q$ is coprime to $|E|$, the representation
theory of $E$ over $\overline{\F}_q$ is paralleled with that over the
complex numbers.)

For $g\in SL_2(p)\leq \widetilde{Qd}(p)$, the
characters of representations  $\phi$ and $\phi^g$ coincide, so $\phi$ and
$\phi^g$ are equivalent. Therefore, there is $h\in GL_p(\overline{\F}_q)$
such that $\phi^g=\phi^h$. As $\phi$ is absolutely irreducible, the
$\F_q$-envelope of $\phi(E)$ is $Mat_p(\F_q)$, and $h$ induces an
automorphism of  $Mat_p(\F_q)$. By the  Skolem-Noether theorem, $h$ can be
chosen in $G=GL_p(q)$. By Schur's lemma, $h$ is unique up to a scalar
multiple. So $g \mapsto h$ is  a projective representation of $SL_2(q)\to
G$. As the Schur multiplier of $PSL_2(p)$ is of order 2, every projective
representation of $SL_2(q)$ arises from an ordinary one, so $h$ can be chosen
so that $g \mapsto h$ is an ordinary representation. If $p > 3$, then
$\widetilde{Qd}(p)$ has no non-trivial Abelian quotient. Since $G/G'$ is
Abelian, it follows that $G'$ contains a subgroup $H$ isomorphic to
$\widetilde{Qd}(p)$. 

Let us now consider the case $G = U_p(q)$. Assume first $p > 3$. By the
previous paragraph, we can assume that $\widetilde{Qd}(p) \cong H \leq
SL_p(q^2)$ and $E \leq G$. It is well known that there exists an
involutive automorphism $\tau$, say, of $GL_p(q^2)$ such that $U_p(q)$ is
exactly the fixed point subgroup of $\tau$. Let $g \in H$,
%\widetilde{Qd}(p)$,
$x \in E$. Then $gxg\up = \tau(gxg\up)=\tau(g)x\tau(g)\up$, whence $g\up
\tau(g) x(g\up \tau(g))\up = x$. As $E$ is absolutely irreducible, by
Schur's lemma, $g\up \tau(g)$ is a scalar matrix, $z_g$, say, so $\tau(g)=
z_g g$.  One easily  observes that the mapping $g \mapsto z_g$ is a
homomorphism  of $H \cong \widetilde{Qd}(p)$ into the group of scalar
matrices of $GL_p(q^2)$. As $\widetilde{Qd}(p)$ is perfect for $p>3$, we
have $z_g=1$, and hence $\tau(g)=g$, that is, $g\in SU_p(q)$.

The above argument has to be refined for $p = 3$. In this case, $GL_3(q^2)$
has a subgroup $H$ isomorphic to $\widetilde{Qd}(3)$. Recall that a Sylow
$3$-subgroup of $GL_3(q^2)$ coincides with one of $U_3(q)$ and hence $H$
can be assumed to have a Sylow $3$-subgroup contained in $U_3(q)$. The
kernel of the mapping $g \mapsto z_g$ as in the previous paragraph contains
both the derived subgroup $H'$ and the Sylow $3$-subgroup of $H$ contained
in $U_3(q)$. As $H$ is generated by these subgroups, $z_g = 1$ follows for
all $g \in H$. Hence  $H \leq U_3(q)$.

Finally, let again $p > 3$. Observe that the centre of $H$
%the subgroup of $G'$ isomorphic to $\widetilde{Qd}(p)$ 
is contained in $Z(G')$. Therefore,
its image in $PSL_p(q)$ (resp. $PSU_p(q)$) is isomorphic to $Qd(p)$.
\qed\end{proof}

%\begin{lem} \label{pslpq}
%Let $p > 3$. Let $G = PSL_p(q)$ (resp.
%$PSU_p(q)$). Suppose $p | q-1$ (resp. $p | q+1$).  Then $Qd(p)$ is a subgroup of $G$.
%\end{lem}
%
%\begin{proof}
%Under the assumptions on $p$ and $q$, $SL_p(q)$ (resp. $SU_p(q)$) contains
%a subgroup $H$ isomorphic to $\widetilde{Qd}(p)$ by Lemma~\ref{uu2}. By the
%construction in Lemma~\ref{uu2}, $Z(H) = Z(E)$ is contained in $Z(SL_p(q))$
%(resp. $Z(SU_p(q))$). Hence the image of $H$ in $G$ is isomorphic to
%$Qd(p)$.
%\qed\end{proof}

Next we examine the case $p=3$ not discussed completely in Lemma \ref{uu2}.

\begin{lem}\label{p3l}
Let $p=3$  and $G=PSL_3(q)$. Suppose that $3|q-1$.
\begin{enumerate}[$(i)$]
\item If $q-1$ is not a multiple of $9$, then the Sylow $3$-subgroups of
$G$ are Abelian, and $G$ has no section isomorphic to $Qd(3)$.

\item If $q-1$ is a multiple of $9$, then $Qd(3)$ is isomorphic to a
subgroup of $G$. Moreover, $SL_2(3)$ has a subgroup isomorphic to
$\widetilde{Qd}^-(3)$ but no one isomorphic to $\widetilde{Qd}(3)$.
\end{enumerate}
\end{lem} 

\begin{proof}
\begin{enumerate}[$(i)$]
\item The order of $G$ is $q^3(q-1)^2(q+1)(q^2+q+1)/3$. One easily observes
that the $3$-part of $|G|$ is 9, so the Sylow $3$-subgroups of $G$ are
Abelian. Then $Qd(3)$ is not a section of $G$. 

\item Assume $9 | q-1$. By Lemma~\ref{uu2}, $GL_3(q)$ contains a subgroup
$X$ isomorphic to $\widetilde{Qd}(3)$ whose image in $PGL_3(q)$ is
isomorphic to $Qd(3)$. Now, $X \cong E \rtimes (Q_8 \rtimes
C_3)$, where $Q_8$ is a quaternion group. Moreover, $X' \cong E \rtimes
Q_8$ is contained in $SL_3(q)$ and $X = X' \rtimes \langle x \rangle$,
where $x^3 = 1$.

Let $3^\theta$ be the 3-part of $q-1$. Then a Sylow 3-subgroup $P$ of
$SL_3(q)$ is isomorphic to $(C_{3^\theta} \times C_{3^\theta}) \rtimes
C_3$. A straightforward calculation shows that any subgroup of $P$ of
exponent 9 is contained in a subgroup isomorphic to $(C_9 \times C_9)
\rtimes C_3$ obtained from $P$ in the obvious way. 
%The latter group is
%isomorphic to a Sylow 3-subgroup of $SL_3(19)$. 
However, this group does not contain a Sylow 3-subgroup
of $\widetilde{Qd}(3)$ (see also Remark~\ref{tildeqd3rem}). Thus $x
\notin SL_3(q)$ and hence
%If $a$ is any
%non-central element of order $3$ of $P$, then the elements of order $3$ of
%$C_P(a)$ are contained in $Z(P) \times \langle a \rangle$. Since the
%element $x$ of $X$ does centralise some non-central element $a \in E$ and
%does not lie in $Z(E) \times \langle a \rangle$, we conclude that $P$ has
%no subgroup isomorphic to a Sylow 3-subgroup of $\widetilde{Qd}(3)$ and
%hence the latter is not a subgroup of $SL_3(q)$.
$\det(x)^3 = 1 \ne \det(x)$. Let $\alpha \in \F_q$ such that
$\alpha^3 = \det(x)$ and set $y = \alpha^{-1}x$. Let $Y = \langle X', y
\rangle$. Then $Y$ is contained in $SL_3(q)$ and the image of $Y$ in
$PGL_3(q)$ is equal to that of $X$ whence the claim on $G$.

Finally, Remark~\ref{tildeqd3rem} implies that $SL_3(q)$ does not contain
a subgroup isomorphic to $\widetilde{Qd}^+(3)$ and hence $Y \cong
\widetilde{Qd}^-(3)$ whence the claim.
\end{enumerate}
\qed\end{proof}

\begin{lem}\label{p3u}
Let $p=3$  and $G=PSU_3(q)$. Suppose that $q+1$ is a multiple of $3$. Then
$Qd(3)$ is  isomorphic to a subgroup of $G$ if and only if $q+1$ is a
multiple of $9$. In this case, $SU_3(q)$ has a subgroup isomorphic to
$\widetilde{Qd}^-(3)$ but no one isomorphic to $\widetilde{Qd}(3)$ or
$\widetilde{Qd}^+(3)$. Otherwise, $Qd(3)$ is not a section of $G$.
\end{lem} 

\begin{proof}
Suppose $9|q+1$. We have shown in the proof of Lemma \ref{p3l} that
$SL_3(q^2)$ contains a subgroup $Y$ such that $Y/Z(Y)\cong Qd(3)$. 
Note that
$
|Z(SL_3(q^2))|=|Z(SU_3(q))|=3.
$
Let $\tau$ be as in the proof of
Lemma~\ref{uu2}, so by the argument there $z_g:=g\up \tau(g)\in
Z(SL_3(q^2))$, and hence $z_g\in SU_3(q)$. Then, applying $\tau$ to
$\tau(g)=z_gg$, we have $g=\tau^2(g)=z_g \tau(g)=z_g^2g$, whence $z_g^2=1$,
$z_g=1$. Therefore, $g=\tau(g)$ and  hence $g\in SU_3(q)$. The statement on
$\widetilde{Qd}^\pm(3)$ follows from Lemmas~\ref{p3l} and \ref{ug1}.

Conversely, let $G=PSU_3(q)$, where $q+1$ is not a multiple of 9. The order
of $G$ is $q^3(q+1)^2(q-1)(q^2-q+1)/3$. One easily observes that the $3$-part
of $|G|$ is 9, so the Sylow  3-subgroups of $G$ are Abelian, whence the
result.
\qed\end{proof}

\begin{lem}\label{co2}
Let $n>p$ and $G=PSL_n(q)$ (resp.,  $PSU_n(q)$), where  $p|q-1$  (resp.
$p|q+1$). Then $G$ contains a subgroup isomorphic to $\widetilde{Qd}(p)$.  
\end{lem}

\begin{proof}
Consider the embedding
$
\nu: SL_p(q) \to SL_n(q),\ \ \ x \mapsto \diag(x,\Id_{n-p}).
$
Then $\nu(SL_p(q))\cap Z(SL_n(q))=1$. This provides an embedding
$SL_p(q) \to PSL_n(q)$. So $G$ has a subgroup isomorphic to
$\widetilde{Qd}(p)$ for $p > 3$.

This can be refined to the case $p=3$ by using the embedding
\[
\mu: GL_3(q) \to SL_n(q),\ \ \ x \mapsto \diag(x, \det x\up , \Id_{n-4}).
\]
If the matrix  $\diag(x, \det x\up, \Id_{n-4})$ is scalar, then either $x =
\Id_n$ or $n=4$ and $x=a\cdot \Id_3\in GL_3(q)$. Moreover, in the latter case
%$\diag(x, \det x\up )=\diag(a,a,a,a^{-3})$, which is scalar only if
$\det x\up = a^{-3}=a$ must hold, so $a^4=1$. As such, if $x \ne \Id$, then
it is not contained in $\widetilde{Qd}(3) \leq GL_3(q)$. Therefore, the
homomorphism $GL_3(q) \to PSL_4(q)$ is faithful on $\widetilde{Qd}(3)$, so
$G$ contains a subgroup isomorphic to $\widetilde{Qd}(3)$.

The proof for the case of unitary groups is similar.
\qed\end{proof}

\begin{lem}\label{hp2}
\begin{enumerate}[$(i)$]
\item Let $g\in GL_n(q)$, $g^p=1\neq g$. Suppose that $g$ is irreducible.
Then $n=e_p(q)$.

\item Let $g\in GL_n(q)$, where $n=e_p(q)$. Then the Sylow $p$-subgroups
of $G$ are cyclic. 

\item Let $2n=e_p(q)$. Then $U_n(q)$ contains an element of order $p$ if
and only if $n$ is odd.
\end{enumerate}  
\end{lem}

\begin{proof}
\begin{enumerate}[$(i)$]
\item It follows from the formula for $|GL_n(q)|$ that $e:=
e_p(q)\leq n$, otherwise $p$ does not divide the group order.
As $g$ is irreducible, the enveloping algebra $[g]$ of $g$ is a field (by
Schur's lemma). In addition, the natural $\F_q GL_n(q)$-module $V$ is of
shape $[g] \cdot v$ for some $v \in V$, so $\dim [g] \geq n$. In fact,
$\dim [g] = n$ as the matrix algebra Mat${}_n(\F_q)$ is well known to
contain no subfield of dimension greater than $n$ over $\F_q$. It follows
that $[g] \cong \F_{q^n}$, and hence  $p$ divides $q^n-1$. By
Lemma~\ref{hu2}, $e$ divides $n$. Then $\F_{q^n}$ contains a subfield $F$
isomorphic to $\F_{q^e}$. As the multiplicative group of $\F_{q^n}$ is
cyclic, we have $g \in F$, and hence $[g] \cong F$, which means $F \cong
\F_{q^n}$, that is, $e = n$.

\item The assumption $n = e_p(q)$ is equivalent to saying that $\F_{q^n}$
contains an element of order $p$, whereas $\F_{q^i}$ for $i<n$ contains no
such element. As  $\F^*_{q^i}$ embeds into $GL_{i}(q)$, it follows that a
subgroup of $GL_n(q)$ isomorphic to $\F^*_{q^n}$ contains a Sylow
$p$-subgroup of $GL_n(q)$, which is cyclic. 

\item Recall that
\[
|U_n(q)| = (q+1) (q^2 - 1) \cdot \ldots \cdot (q^n \pm 1)
\]
according to whether $n$ is even or odd. As $e_p(q) = 2n$, no
term of the form $q^i -1$ in the above formula is divisible by $p$. If some
$q^i + 1$ is divisible by $p$, then so is $q^{2i} - 1$ and hence $2i = 2
e_p(q)$ must hold. Then $i = n$ is an odd number and the claim is proved.
%If $n$ is odd, then $U_n(q)$ contains an element of order $q^n+1$ and
%$p$ divides this number as $p$ divides $q^{2n}-1=q^{e_p(q)}-1$.
%Suppose $n$ is even.
%Note that $2n = e_p(q)$ implies $n = e_p(q^2)$, so $g$ is
%irreducible in $GL_n(q^2)$ by definition of $e_p(q^2)$. However, $U_n(q)$
%with $n$ even contains no irreducible elements.
\qed\end{enumerate}
\end{proof}

\begin{lem}\label{w22}
Let $e=e_p(q)$. 
\begin{enumerate}[$(i)$]
\item  If $n\geq pe$, then $GL_n(q)$ has a subgroup isomorphic to
$\widetilde{Qd}(p)$. If $n > 3$, then this subgroup is contained in
$SL_n(q)$.

\item If $n\geq 2pe$, then $SU_n(q)$ has a subgroup isomorphic to
$\widetilde{Qd}(p)$.

\item If $e$ is even and $n\geq ep$, then $SU_n(q)$ has a subgroup
isomorphic to $\widetilde{Qd}(p)$.

\item If $e\equiv 2\pmod 4$ and $n\geq pe/2$, then $SU_n(q)$ has a subgroup
isomorphic to $\widetilde{Qd}(p)$ except for the case $e = 2$, $p = 3$ and
$n = 3$.
\end{enumerate}
\end{lem}

\begin{proof}
\begin{enumerate}[$(i)$]
\item Suppose first that $n=pe$. Set $Y=GL_p(q^e)$. By Lemma \ref{uu2}, $Y$
contains a subgroup isomorphic to $\widetilde{Qd}(p)$. So it suffices to
show that there is a homomorphism $Y \to GL_n(q)$ faithful on
$\widetilde{Qd}(p)$. First, observe that, viewing $\F_{q^e}$ as a vector
space of dimension $e$ over $\F_q$, we obtain an embedding of $\F_{q^e}$
into Mat${}_e(\F_q)$, which yields an embedding of Mat${}_p(\F_{q^e})$ into
Mat${}_{pe}(\F_q)$. Therefore, $Y= GL_p(q^e)$ embeds into $GL_{pe}(q)$. 

Note that $n = 3$ if and only if $p = 3$, $e = 1$ and $n = ep$.
If $p>3$, then $\widetilde{Qd}(p)$ is perfect, so $\widetilde{Qd}(p)$
embeds into $SL_{pe}(q)$. If $p=3$ and $e > 1$, then $p|q+1$, so $e=2$.
Let $\overline Y$ be the image of $Y$ in $GL_{6}(q)$. Then the index of
$\overline Y \cap SL_{6}(q)$ in $\overline Y$ divides $q-1$. So either
$\widetilde{Qd}(3)$ embeds into $SL_{6}(q)$ or $\widetilde{Qd}(3)$ has a
proper normal subgroup, whose index in $\widetilde{Qd}(3)$ divides $q-1$.
So the index is coprime to $3$, and hence is a 2-power as
$|\widetilde{Qd}(3)|=3^4\cdot 8$. It is well known that $SL_2(3)$, and
hence $\widetilde{Qd}(3)$, has no proper quotient group of 2-power order.
It follows that $SL_6(q)$ has a subgroup isomorphic to $\widetilde{Qd}(3)$.

Finally, let $n > pe$. The case $p = 3$, $e = 1$ has already been handled
in the proof of Lemma~\ref{co2}. Otherwise $SL_n(q)$ has a subgroup
isomorphic to $SL_{pe}(q)$ and $(i)$ follows from the above. 
 
\item Suppose first that $(e,p)\neq (1,3)$. By part $(i)$, $SL_{pe}(q)$
contains a subgroup isomorphic to $\widetilde{Qd}(p)$. There is an
embedding $SL_{pe}(q) \to SU_{2pe}(q)$, whence the result.

Let $e=1$, $p=3$, so $3|q-1$. Then $\widetilde{Qd}(3)$ is a subgroup of
$GL_3(q)$ (see Lemma~\ref{uu2}) and there is an  embedding $GL_{3}(q) \to
U_{6}(q)$. Note that $U_6(q)/SU_6(q)$ is of order $q+1$, which is coprime
to 3. So either $\widetilde{Qd}(3)$ embeds into $SU_{6}(q)$ or
$\widetilde{Qd}(3)$ has a proper normal subgroup, whose index in
$\widetilde{Qd}(3)$ divides $q+1$. So the index is coprime to $3$, and
hence a 2-power as above. As $\widetilde{Qd}(3)$ has no proper quotient
group of 2-power order, it follows that $SU_6(q)$ has a subgroup isomorphic
to $\widetilde{Qd}(3)$.

Consequently, $(ii)$ holds for $n = 2pe$ and hence for $n > 2pe$, too.

\item Let $e$ be even and let $e' = e_p(q^2)$. Then $e' = e/2$. By part
$(i)$, $SL_{pe'}(q^2)=SL_{ep/2}(q^2)$ has a subgroup  isomorphic to
$\widetilde{Qd}(p)$ unless $n = 3$. As there is an embedding
$SL_{pe/2}(q^2) \to SU_{ep}(q)$, the statement follows. For $n = 3$ we
proceed as in part $(ii)$.
 
\item Let $e = 2m$, where $m$ is odd. Then $p$ divides $q^{m}+1$. By
Lemma~\ref{uu2}, $\widetilde{Qd}(p)$ is isomorphic to a subgroup of
$SU_p(q^m)$, provided  $p > 3$. By \cite[Hilfsatz 1]{huppert:1970}, there
is an embedding $SU_{p}(q^{m}) \to SU_{pm}(q)$, whence the result follows
for $p > 3$. If, however, $p = 3$ and hence $e = 2$, then
$\widetilde{Qd}(3)$ is isomorphic to a subgroup of $U_3(q)$ by
Lemma~\ref{uu2}. Since there is an embedding $U_3(q) \to SU_n(q)$ for $n >
3$, the result follows.
\qed\end{enumerate}
\end{proof}

Next we show that if the assumptions of Lemma~\ref{w22} fail, then the Sylow
$p$-subgroups of $G$ are Abelian.

\begin{lem}\label{wn1}
Let $e=e_p(q)$. 
\begin{enumerate}[$(i)$]
\item If $n<ep$, then the Sylow $p$-subgroups of $GL_n(q)$ and hence of
$PSL_n(q)$ are Abelian.

\item If $e$ is odd and $n<2ep$, then the Sylow
$p$-subgroups of $U_n(q)$ and hence of $PSU_n(q)$ are Abelian. 

\item If $e \equiv 0\pmod 4$ and $n < ep$, then the Sylow
$p$-subgroups of $U_n(q)$ and hence of $PSU_n(q)$ are Abelian. 

\item If $e\equiv 2\pmod 4$ and $n<ep/2$, then the Sylow $p$-subgroups of
$U_n(q)$ and hence of $PSU_n(q)$ are Abelian. % \ii $n<2pe'$
\end{enumerate}
\end{lem}

\begin{proof}
\begin{enumerate}[$(i)$]
\item As $|GL_n(q)| = (q-1)\cdot \ldots \cdot (q^n-1)q^a$, the order of a
Sylow $p$-subgroup of $GL_n(q)$ equals the $p$-part of $(q-1)\cdot \ldots
\cdot (q^n-1)$. By Lemma \ref{hu2}, $p$ divides $q^j-1$ if and only if $e$
divides $j$. Therefore, the $p$-part of  $(q-1)\cdots (q^n-1)$ coincides
with that of $(q^e-1)(q^{2e}-1)\cdot \ldots \cdot (q^{ke}-1)$ for some
$k < p$.

We claim that $p$ is coprime to $\frac{q^{ie}-1}{q^e-1}$ for $i<p$. Indeed,
\[
\frac{q^{ie}-1}{q^e-1}=(q^{(i-1)e}-1)+\cdots +(q^e-1)+i,
\]
whence the claim follows. Therefore, if $p^d$ is the $p$-part of $q^e-1$,
then the $p$-part of $|GL_n(q)|$ equals $p^{dk}$, and coincides with that
of $GL_k(q^e)$. In addition, $p^{dk}$ coincides with the $p$-part of the
order of the group of diagonal matrices of $GL_k(q^e)$. Hence the latter
is one of the Sylow $p$-subgroups of $GL_k(q^e)$ and these are Abelian.

Now, there is an embedding $GL_k(q^e) \to GL_n(q)$ and the $p$-parts of the
orders of these groups are the same. So the Sylow $p$-subgroups of
$GL_k(q^e)$ are isomorphic to those of $GL_n(q)$, whence the result.

\item By Lemma~\ref{ug1}, the Sylow $p$-subgroups of $U_n(q)$ are
isomorphic to those of $Gl_l(q^2)$, where $l$ is the integral part of
$n/2$. By assumption $n < 2ep$, so $l < ep$. Moreover, $e_p(q) = e_p(q^2)$
as this number is odd. Therefore, the Sylow $p$-subgroups of $Gl_l(q^2)$
are Abelian by part $(i)$ and the claim follows.

\item We proceed in a similar way as in part $(ii)$. By Lemma~\ref{ug1},
the Sylow $p$-subgroups of $U_n(q)$ are isomorphic to those of $Gl_l (q^2)$
with the same $l$. But now we have $l < ep/2$ and $e_p(q^2) = e_p(q) / 2 =
e/2$, so part $(i)$ applies again and the Sylow $p$-subgroups under
consideration are Abelian.

\item Now the Sylow $p$-subgroups of $U_n(q)$ are isomorphic to those of
$GL_n(q^2)$ and $e_p(q^2) = e/2$, so the assumption $n < ep/2$ ensures that
part $(i)$ can be applied and the result follows.
%
%In cases $(ii)$ and $(iii)$, by Lemma \ref{ug1}, a Sylow $p$-subgroup of
%$PSU_N(q)$ is isomorphic to one of $GL_{l}(q^2)$, where $l$ is the integral part of $n/2$. 
%
%\item It suffices to deal with the case where $n=2ep-1$. Then $l=ep-1$. We
%have seen that $p$ is coprime to $q^i+1$ for $i<2e$. Therefore, $e=2e':=e_p(q^2)$, so  $l <e'p$. By (1), Sylow $p$-subgroups of $GL_{l}(q^2)$ are Abelian, and the same is true for $n<2ep-1$. 
%
%
%
%(4) Suppose first that $e\equiv 2 \pmod 4$, that is,  $e=2m$, where $m$ is odd. By \cite[p. 532]{W55}, Sylow $p$-subgroups of 
%$H$ remains Sylow in $GL_n(q^2)$. Note that $e'=e_p(q^2)=e/2$. 
%%Then $e'=e/2$.??
%So the result follows from (1).
%
 %
 %(3) Write $e=2m$, where $m$ is even. Then $e'=e_p(q^2)=e/2$.  It suffices to prove the statement for $n=ep-1$. Set $l=(n-2)/2$. Then $l=(ep-2)/2=ep'-1$, and the result again follows from (1). 
\qed\end{enumerate}
\end{proof}

%So $n<mp$. It suffices to prove the statement for $n=mp-1$. 
%As above,  $p$ divides $q^{2j}-1$ \ii $2j$  is a multiple of $e$, equivalently, $j$ is a multiple of  $m$.   
%Therefore, the $p$-part of $(q+1)(q^2-1)\cdots (q^n-1)$ coincides with that of  $(q^m+1)(q^{2m}-1)\cdots (q^{mt}-1)$, where $t=p-1$. Note that the order of the group $X:=U_{p-1}(q^m)$ is $(q^m+1)(q^{2m}-1)\cdots (q^{mt}-1)q^b$ for some integer $b>0$, so the order of a \syl of $X$ coincides with that of $H$.
%As $m$ is odd, there is an embedding $X\ra U_{m(p-1)}(q)$ \cite[Hilfsatz 1]{Hu}, so $X$ is isomorphic to a subgroup of $H$. So it suffices to observe that Sylow $p$-subgroups of $X$ are Abelian. This is trivial as a non-Abelian $p$-grouphas no \irr of degree $d$ with $1<d<p$ over any field of characteristic distinct from $p$. 

\begin{prop}
\begin{enumerate}[$(i)$]
\item Let $G = GL_n(q)$ or $U_n(q)$. If the Sylow $p$-sub\-groups of $G$ are
non-Abelian, then $G$ contains a subgroup isomorphic to $\widetilde{Qd}(p)$.

\item Let $G = PSL_n(q)$ or $PSU_n(q)$. If the Sylow $p$-subgroups of $G$
are non-Abelian, then $G$ contains a subgroup isomorphic to
$\widetilde{Qd}(p)$ or $Qd(p)$.
\end{enumerate}
\end{prop}

\begin{proof}
\begin{enumerate}[$(i)$]
\item This follows from Lemmas~\ref{uu2}, \ref{w22} and \ref{wn1}.

\item Suppose first that $p=n=3$ and $3|q-1$ (resp., $3|q+1$). Then the
Sylow $3$-subgroups of  $G$ are Abelian if and only if $q-1$ (resp., $q+1$)
is not a multiple of 9. So in this case the result follows from
Lemmas~\ref{p3l} and \ref{p3u} for $G=PSL_3(q)$ and  $PSU_3(q)$,
respectively.

%Next, suppose that $p=n>3$ and and $p|q-1$ (resp., $p|q+1$). Then
%$SL_p(q)$ (resp., $SU_p(q)$) contains a subgroup isomorphic to
%$\widetilde{Qd}(p)$ and hence $Qd(p)$ is a subgroup of $G$. In these cases
%the Sylow $p$-subgroups of $PSL_p(q)$ (resp., $PSU_p(q)$) are not Abelian,
%whence the result.
%
%If $n>p\geq 3$ and $p|q-1$ (resp., $p|q+1$) then $G$ contains a subgroup
%isomorphic to $\widetilde G$ by Lemma \ref{co2}.
%
%Finally, suppose that $q-1$ (resp., $q+1$) is not divisible by $p$.
%Then the result follows by comparing Lemmas~\ref{w22} and \ref{wn1}.

Assume $p > 3$ or $n \ne 3$. If by Lemma~\ref{wn1} the Sylow
$p$-subgroups of $G$ are non-Abelian, then we are in one of the situations
in Lemma~\ref{w22} whence the result.
\qed\end{enumerate}
\end{proof}

\paragraph{Symplectic groups}

\begin{lem}\label{sp0}
Let $G=Sp_{2n}(q)$ and let $P$ be a Sylow $p$-subgroup of $G$.

\begin{enumerate}[$(i)$]
\item If $e_p(q)$ is odd, then $P$ is isomorphic to a Sylow $p$-subgroup of
$GL_{n}(q)$.

\item If $e_p(q)$ is even, then $P$ is a Sylow $p$-subgroup of
$GL_{2n}(q)$. If, in addition, $e$ divides $2n$, then a Sylow $p$-subgroup
of $G$ is contained in a subgroup isomorphic to $U_{2n/e}(q^{e/2})$.
\end{enumerate}
\end{lem}

\begin{proof}
\begin{enumerate}[$(i)$]
\item Note that $G$ contains a subgroup $X$ isomorphic to $GL_n(q)$.
Recall that
\[
|Sp_{2n}(q)|=(q^2-1)(q^4-1)\cdot \ldots \cdot (q^{2n}-1)q^a
\]
and\[
|GL_n(q)|=(q-1)(q^2-1)\cdot \ldots \cdot (q^n-1) q^b
\]
for some integers $a>b>0$. So the index of $GL_n(q)$ in $Sp_{2n}(q)$ is
equal to $q^{a-b}(q+1)(q^2+1)\cdot \ldots \cdot (q^n+1)$. We show that the
index is coprime to $p$. If $p | q^i + 1$, then $p \mathrel{|}
q^{2i} - 1$. Then, by Lemma \ref{hu2}, $e$ divides $2i$ and hence $i$ as
$e$ is odd. It follows that $p | q^i - 1$, which is impossible since $p$ is
odd.
%is coprime to $p$, unless possibly if $i$ is a
%multiple of $e$. Let $i=te$ for some integer $t\geq 1$. Then $q^{te}+1 =
%q^{te}-1+2$. As $p$ is odd and divides $q^e-1$, it follows that $p$ does
%not divide $q^{te}+1$, as required. (This argument appeared in the proof of
%Lemma~\ref{uu2}.)

\item For the first statement, see \cite[p. 531]{Weir:1955}.

Let $e = e_p(q)$. To prove the second statement, we start by showing that
$G$ contains a subgroup isomorphic to $U_l(q^m)$, where $m=e/2$ and $l=
2n/e$.

Observe first that $Sp_{e}(q)$ contains an element $g$, say, of order $p$
%(see \cite[Satz 5]{huppert:1970})
since $p | q^e - 1 \mathrel{|} |Sp_e(q)|$. Then $g$ is irreducible as an
element of
$GL_{e}(q)$ by the very definition of $e$. As $e|2n$, it follows that the
natural $\F_qG$-module $V$ is a direct sum of $2n/e$ non-degenerate
subspaces of dimension $e$. One observes that there is a homogeneous element
$h\in G$ of order $p$ (in other words, $h=\diag(g, \ldots, g)$ under a
suitable basis of $V$). Then $C_G(h)\cong U_{2n/e}(q^m)$, see for instance
\cite[Lemma 6.6]{emmett:zalesski:2011}.

Furthermore, observe that $p| q^{m}+1$ as $p | q^{2m}-1 = (q^m-1)(q^m+1)$
and $p \ndiv q^m-1$. Note that $p|q^{2i}-1$ implies $e|2i$, and hence
$m|i$. Therefore, the $p$-part of $|G|$ divides
\[
(q^{e}-1)(q^{2e}-1) \cdot \ldots \cdot (q^{2n}-1).
\]
Consider the term
\[
q^{ie}-1=q^{2im}-1=(q^{im}-1)(q^{im}+1)
\]
with $i$ odd. As $p|q^m+1$, and hence $p|q^{im}+1$, we observe that $p$
is coprime to $q^{im}-1$. Similarly, if $i = 2j$ is even, then
\[
q^{ie}-1=q^{2je}-1=(q^{je}-1)(q^{je}+1).
\]
As $p$ divides $q^{je} - 1 = q^{im} - 1$, it is coprime to $q^{je} + 1$.
Therefore, the $p$-part of $|G|$ divides
\[
(q^{m}+1)(q^{2m}-1)(q^{3m}+1)(q^{me}-1) \cdot \ldots \cdot (q^{le} \pm 1)
\]
according to whether $l$ is odd or even.
%if $l$ is even, otherwise
%\[
%(q^{m}+1)(q^{2e}-1)(q^{3m}+1)\cdot \ldots \cdot (q^{lm}+1).
%\]

Recall that
\[
|U_l(q^m)|= q^b(q^m+1)(q^{2m}-1)(q^{3m}+1) \cdot\ldots\cdot (q^{lm}-(-1)^l)
\]
for some integer $b>0$. Therefore, the $p$-part of $|G|$ is equal to that
of $|U_{l}(q^m)|$ and the lemma is proven.
%equals that of
%\[
%(q^{e}+1)(q^{2e}+1) \cdot \ldots \cdot (q^{le}+1)
%\]
%if $l$ is even and
%\[
%(q^{e}+1)(q^{2e}+1) \cdot \ldots \cdot (q^{(l-1)e}+1)
%\]
%otherwise.
%As $p|q^e-1$, we conclude that $p$ is coprime to $q^{ie}+1$ for any integer
%$i$. So the index of $U_{l}(q^m)$ in $G$ is coprime to $p$, whence the result.
\qed\end{enumerate}
\end{proof}  

\begin{prop}\label{sp2}
Let $G=Sp_{2n}(q)$ and set $e = e_p(q)$. The following are equivalent:
\begin{enumerate}[$(1)$]
\item $G$ contains a subgroup  isomorphic to $\widetilde{Qd}(p)$; 

\item a Sylow $p$-subgroup of $G$ is non-Abelian;

\item $n\geq ep$ if e is odd, and $2n\geq ep$ if e is even.
\end{enumerate}
\end{prop}

\begin{proof}
By Lemma \ref{sp0}, the equivalence of (2) and (3) follows from a
corresponding result for $GL_m(q)$ for $m=n$ or $2n$, see Lemmas~\ref{uu2},
\ref{w22} and \ref{wn1}. The implication $(1)\Rightarrow (2)$ is trivial.
If $e$ is odd, then (3) implies (1) by Lemma~\ref{w22} as $GL_n(q)$ is a
subgroup of $G$.

Let $e=2m$ be even, so $p|q^m+1$. Suppose first $2n=pe$. By part $(ii)$ of
Lemma~\ref{sp0}, some Sylow $p$-subgroup of $G$ is contained in a
subgroup $X$ isomorphic to $U_p(q^{m})$. As $p|q^m+1$, by Lemma \ref{uu2},
$X$ contains a subgroup isomorphic to $\widetilde{Qd}(p)$. If $2n>pe$, then
$G$ contains a subgroup isomorphic to $Sp_{pe}(q)$, so the result follows.
\qed\end{proof}

\paragraph{Orthogonal groups}
\begin{lem}\label{o9}
Let $G=O^-_{2n}(q)$ or $O_{2n+1}(q)$, $e=2m$ and $2n=de$, where $d$ is odd.
Then a Sylow $p$-subgroup of $G$ is contained in a subgroup $X$ isomorphic
to $U_d(q^m)$.
\end{lem} 

\begin{proof}
We first show that $G = O^-_{2n}(q)$ contains a subgroup isomorphic to
$U_d(q^ m)$. Note
that $O^-_{e}(q)$ contains an element $g$, say, of order $p$ as $p | q^m+1$
which divides $|O^-_e(q)|$ by the order formula.
%\cite[Satz 3]{huppert:1970}.
Observe that $g$ is irreducible as an element
of $GL_{e}(q)$ by the very definition of $e$. As $e|2n$, it follows that
$V$, the natural $\F_q G$-module, is a direct sum of $d=2n/e$ non-degenerate
subspaces of dimension $e$. As $d$ is odd, these can be chosen of Witt
index 1 (see \cite[2.5.11]{kleidman:liebeck} and use Witt's theorem).  One
observes that there is a homogeneous element $h\in G$ of order $p$ (under a
suitable basis of $V$ we have $h=\diag(g, \ldots, g)$). Then $C_G(h)\cong
U_{d}(q^m)$, see for instance \cite[Lemma 6.6]{emmett:zalesski:2011}. So
$O^-_{2n}(q)$ and hence $O_{2n+1}(q)$ contains a subgroup $X$ isomorphic to
$U_d(q^m)$. 

So it suffices to show that the $p$-part of $G$ does not exceed that of
$U_{d}(q^m)$, and in turn that the $p$-part of $O_{2n+1}(q)$ does not
exceed that of $U_{d}(q^m)$. However, $|SO_{2n+1}(q)|=|Sp_{2n}(q)|$, and
the $p$-part of $|Sp_{2n}(q)|$ equals the $p$-part of $|U_{d}(q^m)|$ by
Lemma~\ref{sp0}. So the result follows.  
\qed\end{proof}

\begin{lem}\label{oo2}
Let $G=O_{2n+1}(q)$, $q$ odd, and let $P$ be a Sylow $p$-subgroup of $G$.
If $e$ is even, then $P$ is isomorphic to a Sylow $p$-subgroup of
$GL_{2n+1}(q)$. If $e$ is odd, then $P$ is isomorphic to a Sylow
$p$-subgroup of $GL_{n}(q)$.
\end{lem}

\begin{proof}
For the first statement  see \cite[p. 532]{Weir:1955}.
Let $e$ be odd. Then $|G|/2$ coincides with  $|Sp_{2n}(q)|$, and $G$
contains a subgroup $X$ isomorphic to $GL_n(q)$.

By Lemma \ref{sp0}, the order of a Sylow $p$-subgroup of $GL_n(q)$
coincides with that of $Sp_{2n}(q)$, and hence with $|P|$. So the result
follows.
\qed\end{proof}

\begin{prop}\label{oo1}
Let $G = O_{2n+1}(q)$ and $e=e_p(q)$. The following are equivalent:

\begin{enumerate}[$(i)$]
\item $G$ contains a subgroup isomorphic to $\widetilde{Qd}(p)$; 

\item a Sylow $p$-subgroup of $G$ is not Abelian;

\item $n\geq ep$ if $e$ is odd, and $n\geq ep/2$ if $e$ is even.
\end{enumerate}
\end{prop}

\begin{proof}
Note that if $q$ is even, then $SO_{2n+1}(q) \cong Sp_{2n}(q)$ and the
result follows from Proposition~\ref{sp2}, so we can assume that $q$ is
odd.

By Lemma \ref{oo2}, the equivalence of $(ii)$ and $(iii)$ follows from a
corresponding result for $GL_m(q)$ for $m=n$ or $2n$, see Lemma \ref{wn1}.
The implication $(i) \Rightarrow (ii)$ is trivial. If $e$ is odd, then
$(iii)$ implies $(i)$ by Lemma \ref{w22} as $G$ has a subgroup isomorphic to
$GL_n(q)$.

Let $e=2m$ be even.
% Suppose first $2n=pe$.
Then a Sylow $p$-subgroup of
$O^-_{pe}(q)$ and of $G$ is contained in a subgroup $X$ isomorphic to
$U_p(q^{m})$ (see Lemma \ref{o9}). As $p|q^m+1$, by Lemma \ref{w22} $(iv)$,
$X$ contains a subgroup isomorphic to $\widetilde{Qd}(p)$. If $2n\geq pe$,
then $G$ contains a subgroup isomorphic to $O^-_{pe}(q)$, so the result
follows. 
\qed\end{proof}

\begin{lem}\label{o84}
Let $G = O^\pm _{2n}(q)$, $n>3$ and let $P$ be a Sylow $p$-subgroup of $G$.
\begin{enumerate}[$(i)$]
\item $P$ is isomorphic to a Sylow $p$-subgroup of $O_{2n+1}(q)$ or of
$O_{2n-1}(q)$. 

\item If $p\ndiv q^{2n}-1$ (equivalently, $e\ndiv 2n$), then $P$ is
isomorphic to a Sylow $p$-subgroup of both $O_{2n+1}(q)$ and $O_{2n-1}(q)$.

\item
%$P$ remains a Sylow $p$-subgroup of $O_{2n+1}(q)$ if and only if $p$
%does not divide the index $|{O_{2n+1}(q):G}|$, which is $q^n+1$ for $G =
%O_{2n}^+(q)$ and $q^n-1$ for $G = O_{2n}^-(q)$. In other words, 
$P$ remains
a Sylow $p$-subgroup of $O_{2n+1}(q)$ if and only if either $e \ndiv 2n$ or
$e|n$ for $G = O_{2n}^+(q)$ and $e \ndiv n$ for $G = O_{2n}^-(q)$. 

\item If $q$ is even, the above statements remain true if one replaces
$O_{2i+1}(q)$ by $Sp_{2i}(q)$ for $i=n,n-1$.
\end{enumerate}
\end{lem}

\begin{proof}
Recall that $p$ divides $q^{i}-1$ if and only if $e$ divides $i$ (see
Lemma~\ref{hu2}).

For $(i)$, see \cite[p. 533]{Weir:1955} or observe that the statement
easily follows  from the formulas for the orders of these three groups.
Recall that
\[
|O^+ _{2n}(q)|
= 2q^{n(n-1)}(q^2-1)\cdot \ldots \cdot (q^{2(n-1)}-1)(q^n-1),
\]
\[
|O^- _{2n}(q)|
= 2q^{n(n-1)}(q^2-1) \cdot \ldots \cdot (q^{2(n-1)}-1)(q^n+1)
\]
and
\[
|O_{2n+1}(q)| = 2q^{n^2} (q^2 - 1) \cdot \ldots \cdot (q^{2n} - 1).
\]
$(ii)$ follows from that the orders of $O_{2n+1}(q)$ and $O_{2n-1}(q)$
differ in a factor $q^{2n-1}(q^{2n}-1)$.

For $(iii)$ observe that $P$ remains a Sylow $p$-subgroup of $O_{2n+1}(q)$
if and only if $p$ does not divide the index $|{O_{2n+1}(q):G}|$, which is
$q^n+1$ for $G = O_{2n}^+(q)$ and $q^n-1$ for $G = O_{2n}^-(q)$. This
happens if either $e \ndiv 2n$ (so $p \ndiv q^{2n}-1$) or $e | n$ for $G =
O_{2n}^+(q)$ and $e \ndiv n$ for $G = O_{2n}^-(q)$.

Finally, $(iv)$ follows from the fact that
$SO_{2n+1}(q) \cong Sp_{2n}(q)$  for $q$ even.
\end{proof}

Lemma~\ref{o84} $(iii)$ together with  Propositions~\ref{sp2} (for $q$ even)
and \ref{oo1} implies:

\begin{prop}\label{oo3}
Let $G = O^\pm _{2n}(q)$. Then $G$ contains no subgroup isomorphic to
$\widetilde{Qd}(p)$ if and only if the Sylow $p$-subgroups of $G$ are
Abelian.
\end{prop}

%\begin{proof}
%One implication is clear, so we shall prove the other implication. If
%$q$ is even, then we are done by Proposition~\ref{sp2}, so let $q$ be
%odd. If the Sylow $p$-subgroups of $O_{2n-1}$ are non-Abelian, then we
%are done by Proposition~\ref{oo1}.
%Observe that the statement may only fail if $q$ is odd, the Sylow
%$p$-subgroups of $O_{2n-1}(q)$ are Abelian and those of $O_{2n+1}$ are
%non-Abelian and isomorphic to those of $G$. By Proposition~\ref{oo1}
%$n = ep$ (for $e$ odd) or $n = ep/2$ (for $e$ even) follows. In both cases,
%$e | 2n$, so by part $(iii)$ of Lemma~\ref{o84} $e|n$ and hence $e$ is odd
%if $G = O_{2n}^+(q)$ and $e \ndiv n$ and hence $e$ is even otherwise.
%However, if $e$ is even, then as in Proposition~\ref{oo1}, $O_{2n}^-(q) =
%O_{pe}^-(q)$ contains a subgroup isomorphic to $\widetilde{Qd}(p)$.
%Finally, if $e$ is odd, so $G = O_{2ep}^+(q)$
%\qed\end{proof}

\begin{proof}
It suffices to show that $G$ contains $\widetilde{Qd}(p)$ if the Sylow
$p$-subgroups of $G$ are non-Abelian. By  Proposition~\ref{oo1}, this is
true if the Sylow $p$-subgroups of $O_{2n-1}(q)$ are non-Abelian.  Assume
that this is not the case. Then, by  Lemma~\ref{o84}(i), the Sylow
$p$-subgroups of $O_{2n+1}(q)$ are non-Abelian, and  Proposition~\ref{oo1}
implies that $n = ep$ (for $e$ odd) or $n = ep/2$ (for $e$ even).  By part
$(iii)$ of Lemma~\ref{o84} we have $G=O_{2n}^+(q) $ if $e $ is odd, and
$G=O_{2n}^-(q) $ if $e$ is even. In the former case $G$ contains $GL_n(q) =
GL_{ep}(q)$ which contains
%\supseteq GL_{p}(q^e)
$\widetilde{Qd}(p)$ by Lemma~\ref{w22}. The latter case has been already
dealt with in the proof of Proposition~\ref{oo1} $(iii)$.
\qed\end{proof}

\begin{prop}\label{po3}
\begin{enumerate}[$(i)$]
\item Let $G = O_{2n}^+(q)$ and let $P$ be a Sylow $p$-subgroup of $G$. If
$e$ is odd, then $P$ is Abelian if and only if $n < ep$. If $e$ is even,
then $P$ is Abelian if and only if $n-1 < ep/2$.

\item Let $G = O_{2n}^-(q)$ and let $P$ be a Sylow $p$-subgroup of $G$. If
$e$ is odd, then $P$ is Abelian  if and only if $n-1 < ep$. If $e$ is even,
then $P$ is Abelian if and only if $n < ep/2$.
\end{enumerate}
\end{prop}

\begin{proof}
Suppose first that $e$ is odd. By Proposition~\ref{oo1}, the Sylow
$p$-subgroups of $G$ are Abelian if $n < ep$ (since those of $O_{2n+1}(q)$
are Abelian). Furthermore, the Sylow $p$-subgroups of $G$ are non-Abelian
if $n > ep$ (since those of $O_{2n-1}(q)$ are so). If, however, $n = ep$,
then $e|n$ and hence by part $(iii)$ of Lemma~\ref{o84} the Sylow
$p$-subgroups of $O_{2n}^+(q)$ are non-Abelian while those of $O_{2n}^-(q)$
are Abelian.

Let now $e$ be even. By Proposition~\ref{oo1}, the Sylow
$p$-subgroups of $G$ are Abelian if $n < ep/2$.
%(since those of $O_{2n}(q)$ are Abelian).
Furthermore, the Sylow $p$-subgroups of $G$ are non-Abelian if  $n >
ep/2$.
%(since those of $O_{2n-1}(q)$ are so).
If, however, $n = ep/2$, then $e|2n$ and $e \ndiv n$, so by part $(iii)$ of
Lemma~\ref{o84} the Sylow $p$-subgroups of $O_{2n}^-(q)$ are non-Abelian
while those of $O_{2n}^+(q)$ are Abelian and the result follows.

 %Then either $e\ndiv 2n$ or $e|n$. If $G =
%O^-_{2n}(q)$, then the index $|G:O_{2n-1}(q)|=q^{n-1}(q^n+1)$ is coprime to
%$p$, so
%$P$ is isomorphic to a Sylow $p$-subgroup of $O_{2n-1}(q)$ in both cases.
%By Lemma~\ref{oo2}, $P$ is isomorphic to a Sylow $p$-subgroup of
%$GL_{n-1}(q)$, which is Abelian if and only if $n-1 < ep$. Let $G =
%O^+_{2n}(q)$. Then $P$ remains a Sylow $p$-subgroup in $O_{2n+1}(q)$
%(see Lemma~\ref{o84}). As above, this is equivalent to $n < ep$.
%
%Suppose now that $e$ is even. If $e|n$, then the inequalities $n-1 < ep/2$
%and $n < ep/2$ are equivalent. Indeed, if $n = ke$ and $n-1 = ke-1 < ep/2$,
%then $n \nless ep/2$ implies $ke = ep/2$, whence $p = 2k$, a
%contradiction.
%
%Let $G = O^-_{2n}(q)$. If $e \ndiv n$, then $P$ is a Sylow $p$-subgroup in
%$O_{2n+1}(q)$, whence $n < ep/2$ by Lemma~\ref{oo1}. If $e|n$, then $P$ is
%isomorphic to a Sylow $p$-subgroup of $O_{2n-1}(q)$, and hence $n-1 <
%ep/2$, equivalently, $n < ep/2$.
%
%Let $G = O^+_{2n}(q)$. If $e|n$, then $P$ is a Sylow $p$-subgroup of
%$O_{2n+1}(q)$, so $n < ep/2$. If $e \ndiv n$, then $P$ is isomorphic to a
%Sylow $p$-subgroup of $O_{2n-1}(q)$. Therefore, $n-1 < ep/2$.
%%, or equivalently, $n < ep/2$.
\qed\end{proof}

\paragraph{Exceptional groups of Lie type}

We first recall that for $p>2$ the Sylow $p$-subgroups of the simple groups
${}^2B_2(q)$, $q>2$
%, and the Tits' group ${}^2F_4(2)'$ 
are Abelian and the group ${}^2B_2(2)$ is soluble. Therefore, these
groups are not to be considered. 

We use information provided in \cite[p. 111]{gorenstein:lyons:83}. For
$p>2$, a Sylow $p$-subgroup $P$ of a simple group $G$ of Lie type  has an
Abelian normal subgroup $A$ and the order of the quotient group $P_W = P/A$
can be computed from the table in \cite[p. 111]{gorenstein:lyons:83}.
In particular, if $P_W=1$, then $P$ is Abelian. 

Write $|G|=q^ab$, where $b$ is coprime to $q$. Let $\Phi_m$ be the $m$-th
cyclotomic polynomial, that is, an (over the rationals) irreducible
polynomial whose roots are precisely the primitive $m$-th roots of unity.
Then $\Phi_m$ divides $x^m-1$ but does not divide $x^i-1$ for $i<m$. The
table in \cite[p. 111]{gorenstein:lyons:83} provides the expressions
of $b = b(G)$ in terms of the $\Phi_m$'s. For instance, for the twisted
group ${}^2E_6(q)$, we have $b = \Phi_1^4\Phi_2^6\Phi_3^2\Phi_4^2\Phi_6^3
\Phi_8\Phi_{10}\Phi_{12}\Phi_{18}$. Write each expression as $\prod_m
\Phi_m^{r_m}$. Let $m_0$ be the least number $m$ such that $p$ divides
$\Phi_m(q)$. In fact, $m_0 = e_p(q)$, but we prefer to keep here notation
of \cite{gorenstein:lyons:83}.
In a given expression for $b$, let $M$ be the set of numbers $m$ of the
form $m = p^km_0$ for some integer $k > 0$ such that $r_m > 0$. Then
$|P_W|=p^{d}$, where $d=\sum _{m\in M}r_m$. In particular, $P_W = 1$ if and
only if $M$ is empty (see \cite[p. 111]{gorenstein:lyons:83}).

We illustrate this with the example $G = {}^2E_6(q)$. If $p>5$, then $M$ is
empty, so $P$ is Abelian. If $m_0=1$ and $p=5$, then again $P$ is Abelian,
but if $m_0=2$, then $|P_W| = 5$. (In this case $P$ is non-Abelian but this
is not explicitly mentioned in \cite{gorenstein:lyons:83}.)   

We first consider the groups of type $E$. The analysis of the table in
\cite[p. 111]{gorenstein:lyons:83} yields the following conclusion:

\begin{lem}\label{go4}
Let $G = E_6(q)$, $E_7(q)$, $E_8(q)$ or ${}^2E_6(q)$ and let $P$ be a Sylow
$p$-subgroup of $G$. 
\begin{enumerate}[$(i)$]
\item $P$ is Abelian if $p>7$ and non-Abelian if $p=3$;

\item if $p=7$, then $P$ is Abelian unless $G = E_7(q)$ and $m_0=1$ or $2$
or $G = E_8(q)$ and $m_0=1$ or $2$;

\item if $p=5$, then $P$ is Abelian unless one of the following holds:
\begin{enumerate}[$(a)$] 
\item $G = E_6(q)$, $m_0 = 1$;

\item $G = {}^2E_6(q)$, $m_0 = 2$; 

\item $G = E_7(q)$, $m_0 = 1$ or $2$;

\item $G = E_8(q)$, $m_0 = 1$, $2$ or $4$.
\end{enumerate}
%\hskip1cm$(iv)$ Sylow $3$-subgroups of H are non-Abelian.
\end{enumerate}
\end{lem}

Note that $m_0 \neq 6$ in case $(d)$ as $m_0 = e_p(q) < p$. 

We have to decide whether $\widetilde{Qd}(p)$ is a subgroup  of $G$
whenever the Sylow $p$-subgroups of $G$ are non-Abelian. 
%The case $m_0=1$ has been settled in Lemma \ref{ex1}, so we assume $m_0>1$, and hence  $H\neq E_6(q)$ in view of Lemma \ref{go4}.  
The following lemma is an extraction from
\cite[Table 5.1]{liebeck:saxl:seitz}.

%\begin{lemma}\label{sss} Group ${}^2E_6(q)$ contains a subgroup isomorphic to a non-Abelian simple factor of $O^-_{10}(q)$. Group $E_7(q)$ contains subgroups isomorphic to a non-Abelian simple factor of $O^+_{12}(q)$ and $PSU_8(q)$.
%Group $E_8(q)$ contains  subgroups isomorphic to a non-Abelian simple factor of each of the following groups $O^+_{16}(q)$.
%\end{lemma}

\begin{lem}\label{ttt}
Let $G = E_6(q)$, $E_7(q)$, $E_8(q)$ or ${}^2E_6(q)$. Suppose that the
Sylow $p$-subgroups of $G$ are non-Abelian. Then $\widetilde{Qd}(p)$ is a
subgroup of $G$. 
\end{lem}

\begin{proof}
We use information from \cite[Table 5.1]{liebeck:saxl:seitz}.
%In view of Lemma \ref{ex1} and Lemma \ref{go4}(1), we assume
%Suppose first  $m_0=1$. If $H\cong E_6$ then $p<7$. By \cite[Table 5.1]{LSS}, $H$ contains a subgroup isomorphic to $SL_6(q)/Z$, where $Z$ is a central subgroup of $SL_6(q)$.
%The natural embedding $SL_5(q)\ra SL_6(q)$ yields an embedding $SL_5(q)\ra SL_6(q)/Z$.
%By Lemma \ref{uu2} fpr $p=5$ and Lemma \ref{co2} for $p=3$ the group $X_1$
%contains a subgroup isomorphic to $\tilde G$, whence the result.   

Suppose first that $G \cong E_6(q)$ (resp., ${}^2E_6(q)$).
Then two primes: $p=3$ and $p=5$ have to be considered.
Set $X = SL_6(q)$ (resp., $X =
SU_6(q)$) and $X_1 = SL_5(q)$ (resp., $X_1 = SU_5(q)$). By
\cite[Table 5.1]{liebeck:saxl:seitz}, $G$ contains a subgroup isomorphic to
$X/Z$, where $Z$ is a central subgroup of $X$. Let first $p = 3$. Then by
Lemma~\ref{w22}, $X$ and $X/Z(X)$ and hence also $X/Z$ contain a subgroup
isomorphic to $\widetilde{Qd}(3)$. Let now $p=5$, so $m_0 = 1$ (resp., 2).
The natural embedding $X_1 \to X$ yields an embedding $X_1 \to X/Z$. By
Lemma \ref{uu2} $X_1$ contains a subgroup isomorphic to
$\widetilde{Qd}(5)$ whence the result.
% for $p=5$ and Lemma~\ref{co2} for $p=3$,
%the group $SL_5(q)$ (if $m_0 = 1$) or $SU_5(q)$
%(if $m_0 = 2$) contains a subgroup isomorphic to $\widetilde{Qd}(3)$,

Suppose now that $G = E_7(q)$. Then $p = 3$, $5$ and $7$ have to be
considered. By \cite[Table 5.1]{liebeck:saxl:seitz}, $G$ has a subgroup $X$
isomorphic to $\Omega^+_{12}(q)$. We use Propositions~\ref{oo3} and
\ref{po3}. Since $n = 6 > 3 = 1 \cdot 3 = 2 \cdot 3/2$ and $6 >5$, $X$
contains subgroups isomorphic to $\widetilde{Qd}(3)$ and
$\widetilde{Qd}(5)$. Let now $p = 7$, so $m_0 = 1$ or $2$.
%Suppose that $G = E_i(q)$, $i=7,8$. As $E_i(q)$ contains subgroups
%isomorphic to $E_6(q)$ and ${}^2E_6(q)$, the result follows for $p=3$ and
%for $p=5$ for $m_0 \leq 2$. Then we are left with $p=7$, $m_0 \leq 2$ and
%$p=5$, $m_0 = 4$. Suppose first that $p = 7$. Then 
By \cite[Table 5.1]{liebeck:saxl:seitz}, $G$ contains subgroups isomorphic
to a central quotient of $SL_8(q)$ and of $SU_8(q)$. Therefore, $G$
contains subgroups isomorphic to $SL_7(q)$ and $SU_7(q)$. So the result
follows from Lemma~\ref{uu2}.  

Finally, let $G = E_8(q)$. Then $G$ has a subgroup isomorphic to
$\Omega^+_{16}(q)$, so we have $n = 8$ in Propositions~\ref{oo3} and
\ref{po3}. Then $ep$ or $ep/2$ in question are $3$ (for $p = 3$), $5$, $5$
and $10$ (for $p = 5$) and $7$ (for $p = 7$). Since only $10$ exceeds $8$,
we are left with the case $p = 5$ and $m_0 = 4$. Again by Table~5.1 in
\cite{liebeck:saxl:seitz}, $G$ has a subgroup isomorphic to $SU_5(q^2)$.
As $m_0 = 4$, $p|q^2 + 1$. So $SU_5(q^2)$, and hence $G$, has a subgroup
isomorphic to $\widetilde{Qd}(5)$. This completes the proof. 
%
% $m_0>1$. Then we have the \f cases.
%
%(i) Let $p=7$, $H=E_8(q)$, $m_0=2=e_p(q)$.  By Propositions \ref{po3} and \ref{oo3},
%a \syl of the group $O^+_{16}(q)$ is non-Abelian, and has a section isomorphic to $G$.  Therefore, $H$ does in view of Lemma \ref{sss}.
%
%Let $p=7$, $H=E_7(q)$, $m_0=2$. By Lemma \ref{w22},  
%$PSU_8(q)$, and hence $H$, has a section isomorphic to $G$.
%
%Let $p=5$, $G = E_8(q)$. By \cite[Table 5.1]{liebeck:saxl:seitz}, $G$
%contains a subgroup isomorphic to
% $O^+_{16}(q)$ and 
%$SU_5(q^2)$.
%Let $m_0 =
%3$. By Propositions~\ref{po3} and \ref{oo3}, the group $O^+_{16}(q)$
%has a subgroup isomorphic to $\widetilde{Qd}(5)$, whence the result.  
%As $m_0 = 4$, $p|q^2 + 1$.
% Note that $\Phi_4(q) = q^2+1$. 
%So $SU_5(q^2)$, and hence $G$, has a subgroup isomorphic to
%$\widetilde{Qd}(5)$. This completes the proof. 
\qed\end{proof} 

%Let $p=5$, $H=E_7(q)$, $m_0=2=e_p(q)$. By Propositions \ref{po3} and \ref{oo3},
%a \syl of the group $O^+_{12}(q)$ is non-Abelian, and has a section isomorphic to $G$.  Therefore, so does $H$ in view of Lemma \ref{sss}.
 
%\medskip
Using \cite[p. 111]{gorenstein:lyons:83}, we conclude that for $p>3$,
the Sylow $p$-subgroups of the groups ${}^3D_4(q)$, $F_4(q)$, ${}^2F_4(q)$
($q=2^{2m+1}$), ${}^2F_4(2)'$, $G_2(q)$, ${}^2G_2(q)$, ($q=3^{2m+1}$) are
Abelian. As we assume that $q$ is not a $p$-power, the groups ${}^2G_2(q)$
for $p=3$ are not to be considered here. 

\begin{lem}\label{og5}
Let $p=3$, $3 \ndiv q$
\begin{enumerate}[$(i)$]
\item If $G = {}^3D_4(q)$, $F_4(q)$, ${}^2F_4(q)$ (with $q=2^{2m+1}$, $m >
0$) or ${^2F}_4(2)'$, then $G$ contains a subgroup isomorphic to $Qd(3)$.

\item If $G = G_2(q)$ and $9 \ndiv q^2-1$, then the Sylow $3$-subgroups of
$G$ are non-Abelian and $G$ contains no section isomorphic to $Qd(3)$.

\item If $G = G_2(q)$ and $9 | q^2-1$, then $G$ contains a subgroup
isomorphic to
 %$\widetilde{Qd}^+(3)$ or 
$\widetilde{Qd}^-(3)$.
\end{enumerate}
%Let $p=3$, $3 \ndiv q$ and let $G = {}^3D_4(q)$, $F_4(q)$, ${}^2F_4(q)$
%(with $q=2^{2m+1}$, $m > 0$), ${^2F}_4(2)'$ or $G_2(q)$. 
%\begin{enumerate}[$(i)$]
%\item $G$ contains a subgroup isomorphic to $\widetilde{Qd}(3)$ if and only
%if $G = {}^3D_4(q)$, $F_4(q)$, $G_2(q)$ with $9|q^2-1 $ or ${}^2F_4(q)$
%with $9 | q+1$.
%
%\item If the Sylow $3$-subgroups of $G$ are non-Abelian, then $G$ contains
%a subgroup isomorphic to $\widetilde{Qd}(3)$ unless $9 \ndiv q^2-1$
%and $G = G_2(q)$ or ${}^2F_4(q)$.
%
%\item If $G = G_2(q)$ and $9 \ndiv q^2-1$, then the Sylow $3$-subgroups of
%$G$ are non-Abelian and $G$ contains no section isomorphic to $Qd(3)$.
%
%\item If $G = {}^2F_4(q)$, $q>2$ and $9 \ndiv q^2-1$, then the Sylow
%$3$-subgroups of $G$ are non-Abelian and $Qd(3)$ is a subgroup of $G$, but
%$\widetilde{Qd}(3)$ is not contained in $G$.
%\end{enumerate}
\end{lem}

\begin{proof}
Let  $G={}^3D_4(q)$. By~\cite[p. 182]{kleidman:1988}, $G$ contains a
subgroup $X$ isomorphic to $PGL_3(q)$ (resp., $PGU_3(q)$) if $3 | q-1$
(resp., $3 | q+1$). By Lemma~\ref{uu2} $X$ has a subgroup isomorphic to
$Qd(3)$ whence the claim.
%If $3|q-1$ (resp. $3|q+1$), then let $Y$ be a subgroup
%of $GL_3(q)$ (resp., $U_3(q)$) containing $SL_3(q)$ (resp., $SU_3(q)$) as a
%subgroup of index 3. Note that $Y$ contains $\widetilde{Qd}(3)$. Indeed,
%$\widetilde{Qd}(3)$ is contained in $GL_3(q)$ (resp., $U_3(q)$) by
%Lemma~\ref{uu2}, and the derived subgroup of $\widetilde{Qd}(3)$ is of
%index 3, whence the claim. Furthermore, $G$ contains a subgroup isomorphic
%to $Y$ (see \cite[Table 5.1]{liebeck:saxl:seitz}), so the result follows.  

Let $G = F_4(q)$. Then $G$ contains a subgroup isomorphic to ${}^3D_4(q)$
(see \cite[Table 5.1]{liebeck:saxl:seitz}), so the result follows from that
for ${}^3D_4(q)$.

Let $G = {^2F}_4(2)'$. Then $G$ contains a subgroup isomorphic to
$PSL_3(3)$ by~\cite{atlas}. By Lemma~\ref{QdpinPSL3}, the latter and hence
$G$ has a subgroup isomorphic to $Qd(3)$.

Let now $G={}^2F_4(q)$, $q=2^{2m+1}>2$, $m > 0$. By Lemma~2.2(6)
in~\cite{malle:1991}, $G$ contains a subgroup isomorphic to ${^2F}_4(2)'$,
so the result follows from the previous paragraph.

Let $G = G_2(q)$. Then there are two maximal subgroups $D_1$, $D_2$ of $G$
with non-Abelian Sylow 3-subgroups; moreover, $D_1$ contains $SL_3(q)$,
$D_2$ contains $SU_3(q)$ as a subgroup of index 2 (see
\cite[Table 5.1]{liebeck:saxl:seitz}).
%Therefore, $G$ contains $Qd(3)$ or
%$\widetilde{Qd}(3)$ as a subgroup or a section if and only if so does
%$SL_3(q)$ or $SU_3(q)$.
If $9|q-1$ (resp., $9|q+1$), then $SL_3(q)$ (resp., $SU_3(q)$) has a
subgroup isomorphic to $\widetilde{Qd}^-(3)$ by Lemmas~\ref{p3l}
and \ref{p3u}.
%Otherwise, by these lemmas, if $3|q-1$
%(resp., $3|q+1$), then the Sylow 3-subgroups of $PSL_3(q)$ (resp.,
%$PSU_3(q)$) are Abelian, and this implies that $Qd(3)$ is not a section of
%$SL_3(q)$ (resp., $SU_3(q)$) as $Z(Qd(3))=1$.
If, however, $9 \ndiv q^2 - 1$, then a Sylow 3-subgroup $E$ of $G$ is
extraspecial of order $27$ and exponent $3$. Therefore, if $Qd(3)$ is
involved in $G$, then it must be involved either in the normaliser of $E$
or in the normaliser of some elementary Abelian subgroup $V$ of $E$. Let
$Z = Z(E)$. Then $N_G(E) \subseteq N_G(Z)$, which has a subgroup of index
2 isomorphic to either $SL_3(q)$ or $SU_3(q)$ according to whether $3
| q-1$ or $3 | q+1$ (see~\cite[p. 461]{flores:foote:2009}). By
Lemmas~\ref{p3l} and \ref{p3u}, these groups do not involve $Qd(3)$. Let us
consider the other case. As $V$ is normal in $E$, it must contain $Z$. Now,
all elements of $E \setminus Z$ are conjugate in $C_G(Z)$ and they are not
conjugate to an element of $Z$ in $G$
(see~\cite[p. 461]{flores:foote:2009}). Thus $N_G(V) \subseteq N_G(Z)$,
which has been proved not to involve $Qd(3)$ whence the claim.

%Note that $3|q+1$. There are two
%maximal subgroups $D_1$, $D_2$ of $G$ with non-Abelian Sylow 3-subgroups,
%each contains the subgroups $SU_3(q)$ and $PU_3(q)$ (with $3|q+1$) as a
%subgroup of index 2 (\cite[Table 5.1]{liebeck:saxl:seitz}). Therefore, if
%$9|q+1$, then $G$ contains a subgroup isomorphic to $\widetilde{Qd}(3)$ by
%Lemmas~\ref{p3l} and \ref{p3u}. Suppose that $q+1$ is not divisible by 9.
%We have seen above that $SU_3(q)$ has no section isomorphic to $Qd(3)$.
%However, by Lemma \ref{uu2}, $U_3(q)$ and hence $PU_3(q)$ contain $Qd(3)$,
%whence the result for $q>2$.
\qed\end{proof}
    
%\medskip 

Thus, we can sum the above arguments to get

\begin{prop}\label{exc}
Let $G$ be a simple group of exceptional Lie type. Suppose that a Sylow
$p$-subgroup of $G$ is not Abelian. If $p>3$, then $\widetilde{Qd}(p)$ is a
subgroup of $G$.

If $p=3$, this is true if $G \cong E_6(q)$, $E_7(q)$, $E_8(q)$ or
${^2}E_6(q)$.
%However, if $G \cong {^3}D_4(q)$, $F_4(q)$, ${^2}F_4(2^m)$
%(with $m>0$) or ${^2}F_4(2)'$, then 
Otherwise $G$ contains a subgroup isomorphic to $Qd(3)$ unless $G \cong
G_2(q)$. In the latter case $G$ contains a subgroup isomorphic to
$\widetilde{Qd}^{-}(3)$ if $9 | q^2 - 1$ and has no section isomorphic to
$Qd(3)$ if $9\ndiv q^2-1$.
%and ${}^2F_4(q)$.
%If $9 \ndiv q^2-1$, then ${}^2F_4(q)$ ($q>2$) and also ${}^2F_4(2)'$
%contain subgroups
%isomorphic to $Qd(3)$, whereas, $G_2(q)$ has no section isomorphic to
%$Qd(3)$. If $9 | q^2 -1$, then $G_2(q)$ contains a subgroup isomorphic to
%$\widetilde{Qd}^-(3)$.
\end{prop}

%%%%%%%%%%%%%%%%%%%%%%%%%%%%%%%%%%%%%%%%%%%%%%%%%%%%%%%%%%%%%%%%%%%%%%%%%%%
\section{The case of the sporadic groups}\label{sporadic}
Having a look at the orders of the sporadic groups, we find only few primes
to consider as a group having a $Qd(p)$-section must have a Sylow
$p$-subgroup of order at least $p^3$. The primes together with the relevant
groups are the following:

\begin{itemize}
\item For $p = 3$: $M_{12}$, $M_{24}$, $J_2$, $J_3$, $J_4$, $Co_1$, $Co_2$,
$Co_3$, $Fi_{22}$, $Fi_{23}$, $Fi'_{24}$, $McL$, $He$, $Ru$, $Sz$, $O'N$,
$HN$, $Ly$, $Th$, $B$, $M$.

\item For $p = 5$: $Co_1$, $Co_2$, $Co_3$, $HS$, $McL$, $Ru$, $HN$, $Ly$,
$Th$, $B$, $M$.

\item For $p = 7$: $Fi'_{24}$, $He$, $O'N$, $M$.

\item For $p = 11$: $J_4$.

\item For $p = 13$: $M$.
\end{itemize}

A non-trivial section of a simple group is a section of one of its maximal
subgroups. In the following examination we use the results listed in the
Atlas of finite simple groups, see \cite{atlas}, or \cite{onlineatlas}.
Since we employ results of the Atlas, it seems to be reasonable to keep
Atlas notation in this section.

\begin{itemize}

\item $p = 3$:

The maximal subgroups of $J_2$ with order divisible by 27 are $U_3(3)$
and $3.A_6.2$. As none of them involves $Qd(3)$, $J_2$ does not either.
Similarly, the only non-soluble maximal subgroup of $J_3$ with the
required order is $(3 \times A_6): 2_2$, which does not involve $Qd(3)$
and hence they are section 3-stable.
% The groups $J_2$ and $J_3$ have already been discussed in
% Section~\ref{solublesimple}.

$M_{12}$ has a maximal subgroup of type $3^2{:}2S_4$. Note that $3^2$ is
self-central\-ising and $2S_4 = GL_2(3)$ here. Hence this maximal subgroup
contains a subgroup isomorphic to $Qd(3)$.
Therefore, the simple groups $M_{12}$, $M_{24}$, $J_4$, $Co_1$, $Co_3$,
$Fi_{22}$, $Fi_{23}$, $Fi'_{24}$, $Sz$, $HN$, $B$, $M$ all contain
subgroups isomorphic to $Qd(3)$ and hence they are non-3-stable.

$McL$ contains a maximal subgroup of type $U_4(3)$. By
Theorem~\ref{defchar}, $U_4(3)$ has a subgroup isomorphic to $Qd(3)$. Hence
each of the groups $McL$, $Co_2$ and $Ly$ contains a subgroup isomorphic to
$Qd(3)$, as they are overgroups of $McL$. Consequently, all these groups
are non-3-stable.

%$He$ has a maximal subgroup of type $2^2.L_3(4).S_3$. This group contains
%a non-split extension $2^2{^.}Qd(3)$. This group is a new example for 
%a minimal non-3-stable group. %????

The derived subgroup of the normaliser of $3A^2$ in $He$ has structure
$(2^2 \times 3^2){^.}SL_2(3)$. This is a non-split extension $2^2{^.}Qd(3)
= 3^2{:}(2^2{^.}SL_2(3))$. This group is a new example for a minimal
non-3-stable group.

$Ru$ has a maximal subgroup of type ${^2F}_4(2)'.2$. By
Proposition~\ref{p11} the Sylow 3-subgroups of he latter are non-Abelian.
Thus by Theorem~\ref{mainLie} ${^2F}_4(2)'.2$ and hence $Ru$ contains a
subgroup isomorphic to $Qd(3)$. As a consequence, $Ru$ is non-3-stable.
%By Proposition~\ref{p11},
%the Sylow 3-subgroups of $U_3(5)$ are non-Abelian. Hence by
%Theorem~\ref{mainLie}, $U_3(5)$ and thus $Ru$ contains a subgroup
%isomorphic to $Qd(3)$.

The Sylow 3-subgroups of $O'N$ are elementary Abelian. Hence $O'N$ has no
section isomorphic to $Qd(3)$ and hence it is section 3-stable.

$Th$ has a maximal subgroup of type $U_3(8){:}6$. By Proposition~\ref{p11},
the Sylow 3-subgroups of $U_3(8)$ are non-Abelian. Thus by
Theorem~\ref{mainLie}, $U_3(8)$ and hence $Th$ contains a subgroup
isomorphic to $Qd(3)$. Therefore, $Th$ is non-3-stable.

\item $p = 5$:

The only maximal subgroup of $HS$ with order divisible by $125$ is
$U_3(5):2$. By Theorem~\ref{defchar}, this group and hence $HS$ have no
section isomorphic to $Qd(5)$. Thus it is section 5-stable.

The only non-soluble maximal subgroup of $McL$ with the required order is
$U_3(5)$, so $McL$ has no section isomorphic to $Qd(5)$ whence it is section
5-stable.

The maximal subgroups of $Co_2$ with adequate order are $McL$ and $HS:2$.
Those for $Co_3$ are $McL.2$, $HS$, and $U_3(5):S_3$. Hence none of these
groups has a section isomorphic to $Qd(5)$, so they are all section
5-stable.

$Co_1$ has a maximal subgroup $5^2{:}2A_5$ which is nothing else but
$Qd(5)$. %Ellen?rizve.
We remark that $Co_1$ has a maximal subgroup $5^{1+2}{:}GL_2(5)$, which has
a subgroup isomorphic to $\widetilde{Qd}(5)$.
As a consequence, $Co_1$ is non-5-stable.

$Ru$ has a maximal subgroup of type $5^2{:}4S_5$, which contains a subgroup
isomorphic to $Qd(5)$ and hence $Ru$ is non-5-stable. %Leellen?rizve.

$Th$ has a maximal subgroup of type $5^2{:}GL_2(5)$. Therefore, $Th$ and
its overgroups, $B$ and $M$ have subgroups isomorphic to $Qd(5)$. Thus they
are non-5-stable.

$HN$ has a maximal subgroup of type $5^2.5_+^{1+2}:4A_5$. Here, $4A_5$
contains $SL_2(5)$, which operates on $5^2$ on the natural way. Hence $HN$
has a subgroup isomorphic to $Qd(5)$, so it is not 5-stable.

$Ly$ has a maximal subgroup of type $G_2(5)$, which has a subgroup
isomorphic to $Qd(5)$ by Theorem~\ref{defchar}. Therefore, $Ly$ is
non-5-stable.

\item $p = 7$:

%According to \cite[p. ]{atlas}, 
$He$ has a maximal subgroup of type $7^2{:}2.L_2(7)$, which is isomorphic
to $Qd(7)$. Hence $He$, $Fi'_{24}$ and $M$ all have subgroups isomorphic to
$Qd(7)$ and they are not 7-stable.

The group $O'N$ has a maximal subgroup of type $L_3(7){:}2$. Hence by
Lemma~\ref{QdpinPSL3}, it also has a subgroup isomorphic to $Qd(7)$ and is
therefore non-7-stable.

\item $p = 11$:

$J_4$ has two maximal subgroups of order divisible by $11^3$. These are
$U_3(11){:}2$ and $11^{1+2}{:}(5\times 2S_4)$. None of them has a section
isomorphic to $Qd(11)$, so $J_4$ has no one either. Therefore, $L_4$ is
section 7-stable.

% In section~\ref{solublesimple}, we have already shown that $J_4$ is
% $Qd(5)$-free.

\item $p = 13$:

%Using the Atlas (\cite[p. 217]{atlas}, 
We find that the monster group $M$
has a maximal subgroup with structure $13^2{:}2L_2(13).4$, so $Qd(13)$ is a
subgroup of $M$ and hence it is not 13-stable.
\end{itemize}

We can summarise the above considerations in the next theorem:

\begin{theorem}\label{sporadicthm}
Let $G$ be a sporadic simple group. Then $G$ is $p$-stable if and only if
it is section $p$ stable. Otherwise, either $G = {He}$, $p = 3$ and $G$
contains a subgroup of type $3^2{:}(2^2{^.}SL_2(3))$ or $G$ contains a
subgroup isomorphic to $Qd(p)$ and one of the following holds:
%$G$ is a member of the following list:
\begin{enumerate}[$(i)$]
\item $G = M_{12}$, $M_{24}$, $J_4$, $Co_1$, $Co_2$, $Co_3$, $Fi_{22}$,
$Fi_{23}$, $Fi'_{24}$, $McL$, $Ru$, $Sz$, $HN$, $Ly$, $Th$, $B$ or $M$ and
$p = 3$;

\item $G = Co_1$, $Ru$, $HN$, $Ly$, $Th$, $B$ or $M$ and $p = 5$;

\item $G = Fi'_{24}$, $He$, $O'N$ or $M$ and $p = 7$;

\item $G = M$ and $p = 13$.
\end{enumerate}
\end{theorem}

%%%%%%%%%%%%%%%%%%%%%%%%%%%%%%%%%%%%%%%%%%%%%%%%%%%%%%%%%%%%%%%%%%%%%%%%%%%
\section{Summary on fusion systems}\label{prelims}
%Throughout this paper, $p$ denotes an odd prime, $P$ a finite $p$-group
%which will play the role of either a Sylow $p$-subgroup of a finite group
%$G$ or the $p$-group a fusion system $\mc F$ is defined on.

In this section we summarise the basic knowledge on fusion systems
especially what we need later.
First of all, we give the definition of a saturated fusion system following
\cite{kessar:linckelmann:2008}. All fusion systems we deal with will be
saturated, so we shall omit the word `saturated' in the sequel.

Let $p$ be a prime and let $P$ be a finite $p$-group. A fusion system
$\mc F$ on $P$ is a category whose objects are the subgroups of $P$ and
whose morphisms are certain injective group homomorphisms which will be
written from the right.

The main example of a fusion system is that of a finite group $G$ with
Sylow $p$-subgroup $P$. If $Q$ and $R$ are subgroups of $P$ such that $Q^g
\leq R$ for some element $g \in G$ (that is, $Q$ is subconjugate to $R$),
then conjugation with $g$ gives rise to a map $c_{g, Q, R}$: $Q \to R$
defined by $x \mapsto g^{-1} x g$ for $x \in Q$. The morphisms in the
fusion system $\mc F_P(G)$ of $G$ on $P$ are exactly these maps so that
\[
\Hom_{\mc F_P(G)} (Q, R) = \{ c_{g, Q, R}\ |\ g \in G \text{ s.\,t. }
Q^g \leq R \}.
\]

The definition of an abstract fusion system $\mc F$ extracts the properties
of $\mc F_P(G)$. To give the exact definition, we need some more notions.

\begin{itemize}
\item A subgroup $Q$ of $P$ is called fully $\mc F$-normalised if $|N_P(Q)|
\geq |N_P (Q\phi)|$ for every morphism $\phi \in \mc F$ with domain $Q$.

\item For an isomorphism $\phi$: $Q \to R$ we let
\begin{multline*}
N_\phi = \{ a \in N_P(Q)\ |\ \exists b \in N_P(R)
\text{ s. t. } (a^{-1} x a)\phi = b^{-1} (x\phi) b \forall x \in R \}.
\end{multline*}
This means that the following diagram commutes:

\begin{center}
\begin{tikzpicture}
\node at (-1.3, 1) {$Q$};
\draw[-latex] (-1.1, 1)--(1.1, 1);
\node[above] at (0, 1) {$\phi$};
\node at (1.3, 1) {$R$};
\draw[-latex] (-1.3, 0.7)--(-1.3, -0.7);
\node[left] at (-1.3, 0) {$c_{a, Q, Q}$};
\draw[-latex] (1.3, 0.7)--(1.3, -0.7);
\node[right] at (1.3, 0) {$c_{b, R, R}$};
\node at (-1.3, -1) {$Q$};
\draw[-latex] (-1.1, -1)--(1.1, -1);
\node[below] at (0, -1) {$\phi$};
\node at (1.3, -1) {$R$};
\end{tikzpicture}
\end{center}

Note that if $\phi$ can be extended to a subgroup $H$ of $N_P(Q)$, then
$H \leq N_\phi$.
% $N_\phi$ is the largest subgroup of $N_P(Q)$ to that $\phi$
%is `possibly extendible'.
\end{itemize}

\begin{Def}[Fusion system]
A {\itshape fusion system on the $p$-group $P$} is a category $\mc F$ with
the subgroups of $P$ as objects. Morphisms are injective group homomorphisms
with the usual composition of functions such that the following hold:
\begin{enumerate}[$(i)$]
%\item for all $Q \leq R \leq P$, the inclusion map $\iota$: $Q \to R$
%is a morphism in $\mc F$;
\item For all $Q$, $R \leq P$ the set $\Hom_P(Q, R)$ consisting of the
$P$-conjugations from $Q$ into $R$ is contained in $\Hom_{\mc F} (Q, R)$.

\item For all morphisms $\phi \in \Hom_{\mc F}(Q, R)$, the isomorphism
$\bar{\phi}$: $Q \to Q\phi$ with $x \mapsto x\phi$ (for all $x \in Q$)
and $\bar\phi^{-1}$: $Q\phi \to Q$, $x\phi \to x$ are morphisms in $\mc F$.

\item $\Aut_P(P)$ is a Sylow $p$-subgroup of $\Aut_{\mc F}(P)$.

\item If $Q$ is fully $\mc F$-normalised, then each morphism $\phi \in
\Hom_{\mc F}(Q, P)$ extends to a morphism $\tilde \phi \in \Hom_{\mc F}
(N_\phi, P)$.
\end{enumerate}
\end{Def}

%The main example of a fusion system is the fusion system $\mc F_P(G)$ of
%a finite group $G$ on a Sylow $p$-Subgroup $P$ of $G$, where
%\[
%\Hom_{\mc F}(Q, R) = \Hom_G (Q, R),
%\]
%the set of maps induced by the conjugation action of $G$ on the subgroups
%of $P$. A straightforward calculation shows that $\mc F_P(G)$ is indeed a
%fusion system.

We now collect some notions concerning fusion systems that we shall use
in this paper.

\begin{itemize}
\item A subgroup $Q$ of $P$ is called {\itshape strongly $\mc F$-closed}
if for all subgroup $R$ of $Q$ and for all morphism $\phi$ with domain
$R$, the image $R\phi$ is contained in $Q$.

\item The {\itshape normaliser} of a fully $\mc F$-normalised subgroup $Q$
of $P$ is the subsystem $\mc N_\mc F(Q)$ of $\mc F$ defined on $N_P(Q)$
such that for $R, T \leq N_P(Q)$ the morphism $\phi \in \Hom_\mc F(R, T)$
is in $\Hom_{\mc N_\mc F(Q)}(R, T)$ if $\phi$ extends to a
morphism $\tilde \phi$: $RQ \to TQ$ such that the restriction $\tilde \phi_Q$
is an $\mc F$-automorphism of $Q$.

\item $Q$ is {\itshape normal} in $\mc F$, denoted by $Q \triangleleft \mc F$,
if $\mc F = \mc N_\mc F(Q)$.

\item If $Q$ is normal in $\mc F$, a {\itshape quotient fusion system}
$\mc F / Q$ can be defined on $P/Q$ with morphisms $\bar \phi$: $T/Q \to R/Q$
induced by morphisms $\phi$: $T \to R$.

\item $\mc F$ is called {\itshape soluble} if there is a sequence
\[
1 = Q_0 < Q_1 < Q_2 < \ldots < Q_r = P
\]
with $Q_{i}/Q_{i-1} \triangleleft \mc F / Q_{i-1}$ for all $1 \leq i \leq r$.

\item $O_p(\mc F)$ is the largest normal subgroup of $P$ that is normal in
$\mc F$.

\item A subgroup $Q$ of $P$ is called {\itshape $\mc F$-centric} if $C_P(Q\phi)$
is contained in $Q\phi$ for all morphisms $\phi$ with domain $Q$.

\item $\mc F$ is said to be {\itshape constrained} if $C_P(O_p(\mc F))
\subseteq O_p(\mc F)$.

\item A {\itshape model} of a constrained fusion system $\mc F$ is a
$p$-constrained and $p'$-reduced group $L$ (so that $C_L(O_p(L)) \subseteq
O_p(L)$ and $O_{p'}(L) = 1$) with Sylow $p$-subgroup $P$ such that
$\mc F = \mc F_P(L)$. Note that each constrained fusion system has a
model which is unique up to isomorphism, see
\cite[Proposition C]{brotoetal}.
\end{itemize}

\section{Definition of {\itshape p}-stability for fusion systems}
\label{secpstabfs}
In this section, we define $p$-stable fusion systems and investigate their
properties.

\begin{Def}\label{stabdef}
Let $\mc F$ be a fusion system on the $p$-group $P$. Then $\mc F$ is
said to be {\itshape $p$-stable} if for all fully $\mc F$-normalised
subgroups $Q$ of $P$ whenever $\chi \in \Aut_{\mc F}(Q)$ satisfies
\begin{equation}\label{stabequation} %\tag{$*$}
(a^{-1})^\chi a (a^{-1})^\chi a^{\chi^2} = 1
\end{equation}
for all $a \in Q$, then $\chi \in O_p(\Aut_\mc F(Q))$.
\end{Def}

Next we prove that Definition~\ref{stabdef} is a generalisation of the
notion of $p$-stability of groups to the case of fusion systems.

\begin{theorem}\label{stabeqpstab}
A group $G$ is $p$-stable if and only if its fusion system $\mc F_P(G)$
on a Sylow $p$-subgroup $P$ of $G$ is $p$-stable.
\end{theorem}

\begin{proof}
Observe first that 
\begin{equation}\label{commeq} %\tag{$**$}
[a, x, x] = x^{-1} a^{-1} x a x^{-1} a^{-1} x^{-1} a x x =
(a^{-1})^x a (a^{-1})^x a^{x^2}.
\end{equation}
Let $\chi \in N_G(Q) / C_G(Q) = \Aut_\mc F(Q)$ be the image of $x$ under
the natural homomorphism. Then $\chi$ satisfies
Equation~(\ref{stabequation}) in Definition~\ref{stabdef} if and only if
$x$ satisfies Equation~(\ref{commeq}). Note that as $x$ ranges over the
elements of $N_G(Q)$, its image $\chi$ ranges over the elements of
$\Aut_\mc F(Q)$ and vice versa.
\qed\end{proof}

\begin{prop}\label{subsyststable}
Let $\mc F$ be a $p$-stable fusion system. Then all subsystems of
$\mc F$ are $p$-stable.
\end{prop}

\begin{proof}
Let $\mc G$ be a subsystem of $\mc F$ on a subgroup $S$ of $P$.
Let $Q$ be a subgroup of $S$. Assume some $\chi \in \Aut_{\mc G}(Q)$
satisfies Equation~(\ref{stabequation}). As $\Aut_{\mc G}(Q) \leq
\Aut_{\mc F}(Q)$, $\chi \in O_p(\Aut_{\mc F}(Q))$ follows. But then
\[
\chi \in O_p(\Aut_{\mc F}(Q)) \cap \Aut_{\mc G}(Q) \leq
O_p(\Aut_{\mc G}(Q)).
\]
So $\mc G$ is $p$-stable.
\qed\end{proof}

We can prove a theorem for fusion systems similar to
Corollary~\ref{localpstab}:

\begin{theorem}
Let $\mc F$ be a fusion system on a $p$-group $P$. Then $\mc F$ is
$p$-stable if and only if $\mc N_{\mc F} (R)$ is $p$-stable for all
non-cyclic fully $\mc F$-normalised subgroups $R$ of $P$.
\end{theorem}

\begin{proof}
One direction is clear by Proposition~\ref{subsyststable}.

To show the converse let $Q \leq P$. Assume $\chi \in \Aut_\mc F(Q)$
satisfies Equation~\ref{stabequation}. If $Q$ is cyclic, then $\chi \in
O_p(\Aut_\mc F(Q))$ automatically follows, so we may assume $Q$ is
non-cyclic. Let $\phi$: $Q \to R$ be an $\mc F$-isomorphism such that
$R$ is fully $\mc F$-normalised. Then $\phi^{-1} \chi \phi \in \Aut_\mc F(R)
= \Aut_{\mc N_\mc F(R)} (R)$ satisfies Equation~\ref{stabequation}.
As $\mc N_\mc F(R)$ is $p$-stable by assumption, $\phi^{-1} \chi \phi$
is contained in $O_p(\Aut_{\mc N_\mc F(R)}(R)) = O_p(\Aut_{\mc F}(R))$.
Since $\Aut_\mc F(Q) = \phi \Aut_\mc F(R) \phi^{-1}$, it follows that
$\chi \in O_p(\Aut_\mc F(Q))$.
\qed\end{proof}

As mentioned before, $p$-soluble groups are $p$-stable for $p > 3$. Now we
examine the relationship between $p$-stability and solubility for fusion
systems.

\begin{lem}\label{Qdpsol}
The fusion system of $Qd(p)$ is soluble.
\end{lem}

\begin{proof}
The Sylow $p$-subgroups $P$ of $Qd(p)$ have structure $V \rtimes
C$, where $V$ is an elementary Abelian group of rank $2$ and $C$ is a
cyclic group of order $p$. Now, $V = O_p(Qd(p))$ and the quotient system
is defined on $C$, a cyclic group, so the sequence
\[
1 = Q_0 < Q_1 = V < Q_2 = P
\]
proves the solubility of $\mc F_{P}(Qd(p))$.
\qed\end{proof}

\begin{prop}\label{solnotpstab}
There are soluble fusion systems which are non-$p$-stable.
\end{prop}

\begin{proof}
The fusion system $\mc F_{P}(Qd(p))$ is soluble by Lemma~\ref{Qdpsol}
and not $p$-stable by Theorem~\ref{stabeqpstab}.
\qed\end{proof}

A counterpart of Proposition~\ref{solnotpstab} is the following:

\begin{theorem}\label{Qdpfreesol}
Let $G$ be a group with Sylow $p$-subgroup $P$. If $Qd(p)$
is not involved in $G$, then the fusion system $\mc F_P(G)$ is soluble.
\end{theorem}

\begin{proof}
Let $G$ be a group not involving $Qd(p)$ and assume the theorem holds for
all groups smaller than $G$. Let $Q = Z(J(P))$, the centre of the
Thompson subgroup\footnote{The Thompson subgroup is the subgroup
of $P$ generated by the Abelian subgroups of $P$ of maximal order}
of $P$. Then the normaliser $N = N_G(Q)$ controls strong fusion by
Theorem~B in~\cite[p. 1105]{glauberman:68}. It follows that
$\mc F_P(G) = \mc F_P(N)$.

Therefore, $Q \triangleleft \mc F_P(N) = \mc F_P(G)$ and hence
\[
\mc F_P(G)/Q = \mc F_{P/Q} (N/Q)
\]
by Theorem 5.20 due to Stancu in~\cite[p.~145]{craven}.

%Since $Q$ is normal in $\mc F_P(N)$, $\mc F_Q(Q)$ is weakly normal in
%$\mc F_P(N)$ by Theorem~5.37 in~\cite[p.~153]{craven} (again due to
%Stancu).
%
%Now, $\mc F_Q(Q)$ is trivially soluble and, by assumption,
$\mc F_P(G)/Q$, being the fusion system of the $Qd(p)$-free group $N/Q$ is
soluble as $|N/Q| < |G|$.
Therefore, $\mc F_P(G)$ is soluble.
%by part $(ii)$ of Lemma~5.90
%in~\cite[pp. 181-182]{craven}. The theorem now follows as $\mc F_P(G) =
%\mc F_P(N)$.
\qed\end{proof}

\section{The maximal subgroup theorem}\label{secthompson}
Our next goal is to prove a fusion theoretic version of Thompson's maximal
subgroup theorem, see in~\cite[p. 295, Thm 8.6.3]{gorenstein}. For this
purpose, we first state and prove a lemma that might have its own interest.

\begin{lem}\label{normfunctor}
Let $\mc N$ be a subsystem of $\mc F$ and assume the subgroup $Q$ of $P$
is normal in $\mc N$. Let $R$ be a fully $\mc F$-normalised subgroup of
$P$ that is $\mc F$-isomorphic to $Q$. Let $\phi$: $N_P(Q) \to N_P(R)$
be an $\mc F$-homomorphism such that $Q\phi = R$. Then $\phi$ induces
an injective functor
\[
\Phi\!: \ \ \mc N \to \mc N_{\mc F}(R)
\]
so that $\mc N$ can be embedded into $\mc N_{\mc F}(R)$.
\end{lem}

\begin{proof}
Note first that such a $\phi$ exists for all $R$ (see e. g.
\cite[Lemma~2.2]{kessar:linckelmann:2008}). For an object $T$ of $\mc N$
we define $\Phi(T) = T\phi$. Observe that $T \leq N_P(Q)$ so this definition
makes sense. Let now $\psi$: $T \to S$ be a morphism in $\mc N$. Then
$\Phi(\psi)$: $\Phi(T) \to \Phi(S)$ is defined as
\[
\Phi(\psi) = \psi^\phi = \phi_T^{-1} \psi \phi_S,
\]
where $\phi_T$ and $\phi_S$ denote the restrictions of $\phi$ to $T$ and
$S$, respectively.

\begin{center}
\begin{tikzpicture}
\node at (-1.3, 1) {$T$};
\draw[-latex] (-1.1, 1)--(1.1, 1);
\node[above] at (0, 1) {$\psi$};
\node at (1.3, 1) {$S$};
\draw[-latex] (-1.4, 0.7)--(-1.4, -0.7);
\node[left] at (-1.4, 0) {$\phi$};
\draw[-latex, dashed] (-1.2, -0.7)--(-1.2, 0.7);
\node[right] at (-1.2, 0) {$\phi_T^{-1}$};
\draw[-latex] (1.3, 0.7)--(1.3, -0.7);
\node[right] at (1.3, 0) {$\phi$};
\node at (-1.3, -1) {$T\phi$};
\draw[-latex, dotted] (-1, -1)--(1, -1);
\node[below] at (0, -1) {$\psi^\phi$};
\node at (1.3, -1) {$S\phi$};
\end{tikzpicture}
\end{center}

We claim $\Phi(\psi)$ is an $\mc N_{\mc F}(R)$-morphism.
Indeed, as $\psi$ is an $\mc N$-morphism and $Q \triangleleft \mc N$,
$\psi$ extends to a morphism $\tilde \psi$: $TQ \to SQ$ with $Q\tilde\psi
= Q$. Now, $TQ \leq N_P(Q)$ and hence $\tilde \psi^\phi$ is defined.
We have $(TQ)\phi = (T\phi) R$ and $(SQ)\phi = (S\phi)R$. By construction
$\tilde \psi^\phi$ extends $\psi^\phi$. Moreover,
\[
R\psi^\phi = R \phi_Q^{-1}\psi \phi = Q \psi\phi = R,
\]
so $\tilde \psi ^\phi$ extends $\psi^\phi$ in the required manner.
Therefore, $\Phi(\psi)$ is indeed a morphism in $\mc N_{\mc F}(R)$.

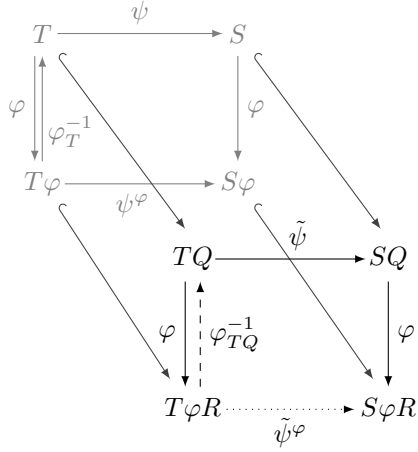
\begin{center}
\begin{tikzpicture}
\begin{scope}
\node[black!50] at (-1.3, 1) {$T$};
\draw[-latex, black!50] (-1.1, 1)--(1.1, 1);
\node[above, black!50] at (0, 1) {$\psi$};
\node[black!50] at (1.3, 1) {$S$};
\draw[-latex, black!50] (-1.4, 0.7)--(-1.4, -0.7);
\node[left, black!50] at (-1.4, 0) {$\phi$};
\draw[-latex, black!50] (-1.3, -0.7)--(-1.3, 0.7);
\node[right, black!50] at (-1.35, -0.3) {$\phi_T^{-1}$};
\draw[-latex, black!50] (1.3, 0.7)--(1.3, -0.7);
\node[right, black!50] at (1.3, 0) {$\phi$};
\node[black!50] at (-1.3, -1) {$T\phi$};
\draw[-latex, black!50] (-1, -1)--(1, -1);
\node[below, black!50] at (-0.1, -1) {$\psi^\phi$};
\node[black!50] at (1.3, -1) {$S\phi$};
\path[right hook-latex, black!75] (1.5, 0.7) edge (3.2, -1.6);
\path[right hook-latex, black!75] (-1.1, 0.7) edge (0.6, -1.6);
\path[right hook-latex, black!75] (1.5, -1.3) edge (3.1, -3.6);
\path[right hook-latex, black!75] (-1.1, -1.3) edge (0.4, -3.6);
\end{scope}

\begin{scope}[xshift=2cm,yshift=-3cm]
\node at (-1.3, 1) {$TQ$};
\draw[-latex] (-1, 1)--(1, 1);
\node[above] at (0.1, 1) {$\tilde\psi$};
\node at (1.3, 1) {$SQ$};
\draw[-latex] (-1.4, 0.7)--(-1.4, -0.7);
\node[left] at (-1.4, 0) {$\phi$};
\draw[-latex, dashed] (-1.2, -0.7)--(-1.2, 0.7);
\node[right] at (-1.2, 0) {$\phi_{TQ}^{-1}$};
\draw[-latex] (1.3, 0.7)--(1.3, -0.7);
\node[right] at (1.3, 0) {$\phi$};
\node at (-1.3, -1) {$T\phi R$};
\draw[-latex, dotted] (-0.85, -1)--(0.85, -1);
\node[below] at (0, -1) {$\tilde\psi^\phi$};
\node at (1.3, -1) {$S\phi R$};
\end{scope}
\end{tikzpicture}
\end{center}

It is straightforward that $\Phi$ preserves compositions and also that
$\Phi$ is injective.
\qed\end{proof}

% Denote the set of self-centralising normal (Abelian) subgroups of $P$ of
% rank at least $3$ by $SCN_3(P)$. Assume this set is non-empty for $P$.
% According to \cite[p. 294]{gorenstein:80}, we define
% \[
%  A_1(P) = \{ Q \leq P \ |\ Q \text{ contains an element of } SCN_3(P) \}
% \]
% and, for $i > 1$
% \begin{multline*}
%  A_i(P) = \{ Q \leq P \ |\ Q \geq R \cong C_p^2 \text{ such that }\\
%  C_P(x) \in A_{i-1}(P) \text{ for all } x \in R \setminus \{1\} \,\}.
% \end{multline*}
% Last, let
% \[
% \mathfrak Q = \bigcup_{i \geq 1} A_i(P).
% \]

\begin{theorem}[Maximal subgroup theorem]\label{thompsonsmaximal}
Let $\mc F$ be a fusion system defined on the $p$-group $P$.
Let $\mathfrak Q$ be a non-empty collection of non-trivial subgroups of $P$
satisfying the following property:
\begin{enumerate} %[$(i)$]
% \item If $Q \in \mathfrak Q$ and $Q \leq R$, then $R \in \mathfrak Q$.
\item[] If $Q \in \mathfrak Q$, and $\phi$: $Q \to R$ is an $\mc
F$-homomorphism, then $R \in \mathfrak Q$.
\end{enumerate}
Set
\[
\mathfrak N = \{ \mc N_{\mc F}(R)\ |\ 1 < R \leq P,\ R
\text{ fully $\mc F$-normalised and } N_P(R) \in \mathfrak Q \}.
\]
Assume each element of $\mathfrak N$ is constrained and $p$-stable.
Then $\mathfrak N$ has a unique maximal element.
\end{theorem}

\begin{proof}
We prove that each element of $\mathfrak N$ is contained in
$\mc M = \mc N_\mc F(Z(J(P)))$.
%, where $J(P)$ denotes the Thompson subgroup of $P$
First assume $R \triangleleft P$. Then $\mc N_{\mc F}(R)$ is
defined on $P$. As $\mc N_{\mc F}(R)$ is constrained and $p$-stable
by assumption, it has a model $L$ which is $p$-constrained, $p'$-reduced
and $p$-stable. Then $C_L(O_p(L)) \subseteq O_p(L)$ and Theorem~A
of \cite{glauberman:68} applies. Therefore, $Z(J(P))$ is normal in
$L$, whence $Z(J(P)) \triangleleft \mc N_{\mc F}(R)$. So $\mc N_{\mc F}(R)
\subseteq \mc M$.

Let now $R \ntriangleleft P$ and assume $\mc N_{\mc F}(S) \subseteq
\mc M$ for all fully $\mc F$-normalised subgroups $S$ of
$P$ satisfying $N_P(S) \in \mathfrak Q$ and $|N_P(S)| > |N_P(R)|$. Now,
$\mc N_{\mc F}(R)$ is defined on $N_P(R)$ and by the above argument
$Z = Z(J(N_P(R))) \triangleleft \mc N_{\mc F}(R)$. Let $Z^*$ be a fully
$\mc F$-normalised subgroup of $P$ that is $\mc F$-isomorphic to $Z$.
Let $\phi$: $N_P(Z) \to N_P(Z^*)$ be an $\mc F$-morphism. By Alperin's
fusion theorem to fusion systems there is a sequence
\[
N_P(Z) = S_0 \sim S_1 \sim \ldots \sim S_t \sim S_{t+1} \subseteq N_P(Z^*)
\]
of subgroups of $P$, there are fully $\mc F$-normalised (and essential)
subgroups $L_1$, $\ldots$, $L_t$ of $P$ such that $S_{i-1}$, $S_i \leq L_i$
for all $1 \leq i \leq t$, there are morphisms $\alpha_i \in
\Aut_{\mc F}(L_i)$ with $S_{i-1} \alpha_i = S_i$ (for all $1 \leq i \leq
t$) and there is a morphism $\sigma \in \Aut_\mc F(P)$ such that $\phi =
\alpha_1 \alpha_2 \ldots \alpha_t \sigma$. Now,
\[
|N_P(L_i)| \geq |L_i| \geq |S_i| = |N_P(Z)| > |N_P(R)|
\]
as $Z$ is characteristic in $N_P(R) < P$. Moreover, $L_i$ contains $S_i$,
a subgroup of $P$ which is $\mc F$-isomorphic to $N_P(Z)$. Hence $L_i \in
\mathfrak Q$. Therefore, by assumption $\mc N_{\mc F}(L_i) \subseteq \mc M$
holds for all relevant $i$. Observe that $\sigma \in \mc M$ is trivial.
Thus
\[
\phi = \alpha_1 \ldots \alpha_t \in 
\mc M
\]
also holds.

By Lemma~\ref{normfunctor} for each $\psi \in \mc N_{\mc F}(R)$
we have $\psi^\phi \in \mc N_{\mc F}(Z^*)$, because $Z$ is normal in
$\mc N_{\mc F}(R)$. Now, $|N_P(Z^*)| \geq |N_P(Z)| > |N_P(R)|$ and by
construction $N_P(Z^*) \in \mathfrak Q$. Hence $\psi^\phi \in
\mc N_{\mc F}(Z^*) \subseteq \mc M$ by assumption. Therefore,
\[
\psi = \phi_T \psi^\phi \phi_S^{-1} \in \mc M
\]
and so $\mc N_{\mc F}(R) \subseteq \mc M$ which proves the theorem.
\qed\end{proof}

Theorem~\ref{thompsonsmaximal} has the following consequence:

\begin{prop}\label{Opne1}
Let $\mc F$ be a fusion system and assume $\mc N_\mc F(Q)$ is constrained
and $p$-stable for all fully $\mc F$-normalised subgroups $Q \ne 1$ of $P$.
Then $Z(J(P)) \triangleleft \mc F$, so $O_p(\mc F) \ne 1$ and hence $\mc F$
is constrained and $p$-stable.
\end{prop}

\begin{proof}
Let $Z = Z(J(P))$. With the set $\mathfrak Q = \{ 1 < Q \leq P \}$
%\ |\ Q$ is fully $\mc F$-normalised $\}$ 
the conditions of Theorem~\ref{thompsonsmaximal} are
certainly satisfied. Hence $\mc N_\mc F(Z)$ is the unique maximal
element of the set
\[
\mathfrak N = \{ \mc N_{\mc F}(R)\ |\ 1 < R \leq P,\ R
\text{ fully $\mc F$-normalised} \}.
\]
%defined in Theorem~\ref{thompsonsmaximal}.
We show $\mc F = \mc N_\mc F(Z)$. To this end, let $\phi$: $T \to S$ be
a morphism in $\mc F$. By Alperin's fusion theorem, there are subgroups
\[
T = T_0 \sim T_1 \sim \ldots \sim T_t \sim T_{t+1} = T\phi \leq S
\]
of $P$ and for all $i = 1$, $\ldots$, $t$, there are fully
$\mc F$-normalised essential subgroups $L_i \leq P$ with $T_{i-1}$, $T_i
\leq L_i$ and automorphisms $\tau_i \in \Aut_\mc F(L_i)$ with $T_{i-1}
\tau_i = T_i$ and an automorphism $\sigma \in \Aut_\mc F(P)$ such
that $\phi = \tau_1 \tau_2 \ldots \tau_t \sigma$. By assumption, for each
$1 \leq i \leq t$ we have
\[
\tau_i \in \mc N_\mc F(L_i) \subseteq \mc N_\mc F(Z)
\]
as $L_i \ne 1$ is fully $\mc F$-normalised. On the other hand, $\sigma
\in \mc N_\mc F(Z)$ trivially holds.
It follows then that $\phi \in \mc N_\mc F(Z)$ and hence $Z \triangleleft
\mc F = \mc N_\mc F(Z)$, whence $O_p(\mc F) \supseteq Z \ne 1$.
%Then by Proposition~\ref{essmax} $Z(J(P)) \triangleleft \mc F$, so
%$O_p(\mc F) \ne 1$, in fact, it contains $Z(J(P))$. Moreover, $\mc F$ is
%constrained (and $p$-stable) since it is the normaliser of $Z(J(P))$.
\end{proof}

Concerning groups, we have the following corollary:

\begin{cor}\label{conststaballO_pne1}
% Let $\mc F = \mc F_P(G)$ for some $p$-stable group $G$ with Sylow
% $p$-subgroup $P$. If $\mc N_\mc F(Q)$ is constrained for all fully
% $\mc F$-normalised subgroups $Q \ne 1$ of $P$, then a local subgroup $N_G(R)$
% of $G$ controls strong fusion in $P$ for some subgroup $R \ne 1$ of $P$.
% 
Let $G$ be a $p$-stable group with Sylow $p$-subgroup $P$. Assume all
$p$-local
subgroups $N_G(Q)$ of $G$ (with $Q \ne 1$) are $p$-constrained. Then the
subgroup $N_G(Z(J(P)))$ controls strong fusion in $P$.
\end{cor}

\begin{proof}
Let $\mc F = \mc F_P(G)$. Then $\mc N_\mc F(Q) = \mc F_{N_P(Q)}(N_G(Q))$ is
$p$-stable and constrained for all non-trivial fully $\mc F$-normalised
subgroups of $P$. Hence Proposition~\ref{Opne1} applies, so $O_p(\mc F)
\geq Z(J(P)) \triangleleft \mc F$. As $Z(J(P))$ is fully $\mc F$-normalised,
$\mc F = \mc N_\mc F(Z(J(P)))$ is the fusion system of $N_G(Z(J(P)))$,
that is, $N_G(Z(J(P)))$ controls strong fusion in $P$.
\qed\end{proof}

\begin{rem}
\begin{enumerate}[$(i)$]
\item The assumptions in Proposition~\ref{Opne1} and
Corollary~\ref{conststaballO_pne1} are strict in the following sense: The
condition that the normaliser systems (or the normalisers in the group)
are $p$-stable cannot be omitted even if it is assumed that the normalisers
are soluble (instead of being constrained). Let namely $G = L_3(3)$, $P$ a
Sylow $3$-subgroup of $G$. Then $G$ is a minimal simple group so that the
local subgroups of $G$ are soluble and hence so are the normaliser systems in
$\mc F_P(G)$. However, the fusion system $\mc F_P(G)$ has no non-trivial
normal subgroups it follows from Theorem~1.2
in~\cite[p. 455]{flores:foote:2009}.
% , see also Section~\ref{solublesimple} of the present paper.

% \item If $G$ is a minimal simple group with Sylow $p$-subgroup $P$, where
% $p > 3$, then $\mc F = \mc F_P(G)$ is soluble. In this case, all local
% subgroups of $G$ are soluble and hence also $p$-stable and $p$-constrained.
% It follows then that $\mc N_\mc F(Q)$ is $p$-stable and constrained for all
% fully normalised subgroups $Q$ of $P$. Then, by Proposition~\ref{Opne1}
% $O_p(\mc F) \ne 1$ and hence $\mc F = \mc N_\mc F(O_p(\mc F))$ is soluble.

\item If $G$ is $p$-soluble (for $p > 3$), then Theorem~C
in~\cite[p. 1105]{glauberman:68} asserts that $N_G(Z(J(P)))$ controls strong
fusion in $P$. It follows from the results of Sections~\ref{Qdsecdefchar}-%
\ref{Qdsecnondefchar}
that the fusion system of a finite simple group $G$ is soluble if and only if
$Z(J(P)) \triangleleft \mc F_P(G)$, that is, if and only if $N_G(Z(J(P)))$
controls strong fusion in $P$. The same is not true in general: the fusion
system of $Qd(p)$ is soluble. For its Sylow $p$-subgroup $P$ we have $J(P)
= P$, so $Z(J(P)) = Z(P)$ has order $p$. Its normaliser is the subgroup $V$
of order $p^2$ (see Example~\ref{qdp}) which certainly does not control the
fusion in $Qd(p)$.
\end{enumerate}
\end{rem}

\section{On {\itshape Qd}({\itshape p})-free fusion systems}
\label{secQdpfree}

For groups, there is a strong connection between $p$-stability and not
involving $Qd(p)$. A corresponding notion for fusion systems is defined
in~\cite[Def.~1.1]{kessar:linckelmann:2008}.

Let $Q$ be a fully $\mc F$-normalised $\mc F$-centric subgroup of $P$.
We examine the normaliser $\mc N = \mc N_{\mc F}(Q)$ of $Q$ in $\mc F$. We
claim $\mc N$ is constrained. Indeed, $Q \leq O_p(\mc N)$, so
\[
O_p(\mc N) \geq Q \geq C_P(Q) \geq C_P(O_p(\mc N))
\]
as $Q$ is $\mc F$-centric. Therefore, $\mc N$ has a model.

\begin{Def}\label{Qdpfree}
Let $\mc F$ be a fusion system on the $p$-group P. $\mc F$ is called
{\itshape $Qd(p)$-free} if $Qd(p)$ is not involved in the models of
$\mc N_{\mc F}(Q)$, where $Q$ runs over the set of $\mc F$-centric fully
$\mc F$-normalised subgroups of $P$.

We shall also call a group {\itshape $Qd(p)$-free} if it does not involve
$Qd(p)$.
\end{Def}

\begin{rem}
Though it is not stated explicitly there, it follows
from~\cite{kessar:linckelmann:2008} that a $Qd(p)$-free
fusion system $\mc F$ is soluble. Indeed, Theorem~B asserts that $Z(J(P))$
is normal in $\mc F$. Now, by Proposition~6.4, $\mc F/Z(J(P))$ is also
$Qd(p)$-free. Since $Z(J(P))$ is non-trivial, the claim follows by
induction.

As the next example shows, a soluble fusion system need not be $Qd(p)$-free.
\end{rem}

\begin{example}
The fusion system of $Qd(p)$ is not $Qd(p)$-free: the subgroup $V$ (as in
Example~\ref{qdp}) is certainly fully $\mc F$-normalised and $\mc F$-centric,
its normaliser is the whole fusion system. The model of the fusion system
is the group $Qd(p)$ itself, being $p$-constrained and $p'$-reduced.
\end{example}

Being soluble, a $Qd(p)$-free fusion system $\mc F$ is constrained and
hence it has a model. By definition, a model of $\mc F =
\mc N_\mc F(O_p(\mc F))$ is $Qd(p)$-free. %Our next result shows that 
Not only is a model of $\mc F$ $Qd(p)$-free, but also every group $G$ such
that $\mc F = \mc F_P(G)$ is $Qd(p)$-free, as the next result shows.

\begin{theorem}\label{freeinv}
Let $G$ be a group, $P$ a Sylow $p$-subgroup of $G$ and $\mc F = \mc F_P(G)$
the fusion system of $G$ on $P$. Then $\mc F$ is $Qd(p)$-free if and only
if $G$ does not involve $Qd(p)$.
\end{theorem}

In order to prove this theorem, we need some preparation.

\begin{Def}\label{p-centricdef}
A $p$-subgroup $Q$ of $G$ is called {\itshape $p$-centric} if every
$p$-element centralising $Q$ is contained in $Q$.
\end{Def}

Note that $Q$ is $p$-centric if and only if $C_P(Q) \leq Q$ for
{\slshape all} Sylow $p$-subgroups $P$ of $G$ containing $Q$. In this case,
$Z(Q)$ is a Sylow $p$-subgroup of $C_G(Q)$ and by Burnside's normal
$p$-complement theorem it follows that $C_G(Q) = Z(Q) \times O_{p'}(C_G(Q))$.

\begin{lem}\label{F=pcentric}
Let $G$ be a group with Sylow $p$-subgroup $P$ and let $\mc F = \mc F_P(G)$
be its fusion system on $P$. Let furthermore $Q$ be a fully normalised
subgroup of $P$. Then $Q$ is $\mc F$-centric if and only if it is
$p$-centric.
%%%% Kell a fully normalised hozzá????
\end{lem}

\begin{proof}
$Q$ is $\mc F$-centric if and only if $C_P(Q^t) \subseteq Q^t$ holds
whenever $Q^t \leq P$. This means that $Q \supseteq C_{P^*}(Q)$ for all
Sylow $p$-subgroups $P^*$ of $G$ containing $Q$. This is equivalent to
saying that $Q$ is $p$-centric.
\qed\end{proof}

\begin{lem}\label{op'factor}
Let $G$ be a group, $P$ a Sylow $p$-subgroup of $G$. Then
\[
\mc F_P(G) = \mc F_P(G/O_{p'}(G)).
\]
Here, we identify $O_{p'}(G)P / O_{p'}(G)$ with $P$.
\end{lem}

\begin{proof}
%For an element $g \in G$, denote the coset $g Q_{p'}(G)$ by $\bar g$.
%Similarly, for a subgroup $H \leq G$ denote the factor $H O_{p'}(G) /
%O_{p'}(G)$ by $\bar H$.
Denote images in $\bar G = G/O_{p'}(G)$ by bar.
The assignment $c_{g, Q, R} \mapsto c_{\bar g, \bar Q, \bar R}$ defines
a map $\mc F_P(G) \to \mc F_{\bar P} (\bar G)$. We have to show it is a
bijection.

We first prove it is surjective. Let $\bar Q$, $\bar R \leq \bar P$
and $\bar g \in \bar G$ such that $\bar Q^{\bar g} \leq \bar R$. Then
conjugation by $g$ maps $Q$ into $RO_{p'}(G)$ and hence $Q^{gt} \leq R$
for some $t \in O_{p'}(G)$. Therefore, the image of $c_{gt, Q, R}$ is
$c_{\bar g, \bar Q, \bar R}$ and surjectivity is proved.

To prove injectivity, assume $c_{\bar g, \bar Q, \bar R} =
c_{\bar h, \bar S, \bar T}$. Then, first of all, $Q = S$ and $R = T$ as
$P$ maps isomorphically to $\bar P$. By the same reason, the operation of
$g$ and $h$ coincides on $Q$. Thus $c_{g, Q, R} = c_{h, S, T}$ and
injectivity is proven.
%
%We have to show that two elements of $G$ lying in the same coset modulo
%$O_{p'}(G)$ induce the same morphisms in $\mc F_P(G)$.
%Let $g \in hO_{p'}(G)$ and assume $Q$, $Q^g, Q^h \leq P$. Then $t = h^{-1}g
%\in O_{p'}(G)$ conjugates $Q^h$ to $Q^g$. For any $q \in Q^h$
%\[
%[t, q] = t^{-1} \cdot t^q = q^{-t} \cdot q \in O_{p'}(G) \cap P = 1.
%\]
%Hence $t$ centralises $Q^h$, so the operations of $g$ and $h$ from $Q$ to
%$Q^g = Q^h$ coincide.
\qed\end{proof}

\begin{prop}\label{modelofN}
Let $\mc F = \mc F_P(G)$. Furthermore, let $Q$ be a fully
$\mc F$-nor\-malised and $\mc F$-centric subgroup of $P$. Then the model of
$\mc N_{\mc F}(Q)$ is isomorphic to $N_G(Q) / O_{p'}(N_G(Q))$.
\end{prop}

\begin{proof}
We prove that the group $L = N_G(Q) / O_{p'}(N_G(Q))$ satisfies the three
conditions on a model. First of all, a Sylow $p$-subgroup of $N_G(Q)$ is
$N_P(Q)$ as $Q$ is fully $\mc F$-normalised. The fusion system of $N_G(Q)$
on $N_P(Q)$ is $N_{\mc F}(Q)$ by Theorem~4.27 in~\cite[p.108]{craven}. Now,
the fusion system of $N_G(Q)$ is the same as that of $L$
%$N_G(Q) / O_{p'}(N_G(Q))$ 
by Lemma~\ref{op'factor}.

Obviously, $L$
%$N_G(Q)/O_{p'}(N_G(Q))$ 
is $p'$-reduced by construction.

It only remained to show that $L$
%$N_G(Q) / O_{p'}(N_G(Q))$ 
is $p$-constrained, that is,
\[
%C_{N_G(Q)/O_{p'}(N_G(Q))}\big(O_p\big(N_G(Q)/O_{p'}(N_G(Q))\big)\big) \leq
%O_p\big(N_G(Q)/O_{p'}(N_G(Q))\big).
C_L(O_p(L)) \leq O_p(L).
\]
%First observe that $Q O_{p'}(N_G(Q)) / O_{p'}(N_G(Q)) \leq
%O_p(N_G(Q)/O_{p'}(N_G(Q)))$ 
Denote the image of $Q$ in $L$ by $\bar Q$. Then $\bar Q \leq O_p(L)$
as $Q$ is normal in $N_G(Q)$, so $C_L(O_p(L)) \leq C_L(\bar Q)$.
%the centraliser on the left-hand side is certainly contained in
%$C_{N_G(Q)/O_{p'}(N_G(Q))} \big( Q O_{p'}(N_G(Q)) / O_{p'}(N_G(Q)) \big)$.

Assume $c O_{p'}(N_G(Q))$ is contained in $C_L(\bar Q)$ for some
$c \in N_G(Q)$. Then $[c, x] \in O_{p'}(N_G(Q))$ for all $x \in Q$. But
$[c, x] = x^{-c} x \in Q$, so it must be equal to $1$ and hence $c$
centralises $Q$. Now, $C_{N_G(Q)}(Q) = C_G(Q) = O_{p'}(C_G(Q)) \times Z(Q)$
as $Q$ is $p$-centric by Lemma~\ref{F=pcentric}. As $O_{p'}(C_G(Q)) \leq
O_{p'}(N_G(Q))$, we have $C_L(\bar Q) = Z(\bar Q)$ and hence
\[
%c O_{p'}(N_G(Q)) \subseteq Z(Q) O_{p'}(N_G(Q)) \subseteq Q O_{p'}(N_G(Q))
C_L(O_p(L)) \leq C_L(\bar Q) \leq \bar Q \leq O_p(L).
\]
whence the claim follows.
\qed\end{proof}

\begin{lem}\label{plocal}
Let $G$ be a group. $G$ involves $Qd(p)$ if and only if $N_G(Q)$ also does
for an appropriate non-cyclic $p$-subgroup $Q$ of $G$.
\end{lem}

\begin{proof}
Assume $G$ involves $Qd(p)$, so there are $K \triangleleft H \leq G$ such
that $H/K = V \rtimes S \cong Qd(p)$. Here, $V$ is an elementary
Abelian group of order $p^2$ and $S \cong SL_2(p)$. Let $\tilde V$ be
a Sylow $p$-subgroup of the preimage of $V$ under the natural homomorphism
$H \to H/K$. Then $K\tilde V \triangleleft H$ is the preimage of $V$ and
hence $H = K \tilde V N_H(\tilde V) = K N_H(\tilde V)$ by Frattini
argument. Now,
\[
Qd(p) \cong H/K = K N_H(\tilde V)/K \cong
N_H(\tilde V) / N_H(\tilde V) \cap K
\]
by the second isomorphism theorem. Therefore, $N_H(\tilde V)$ and so
$N_G(\tilde V)$ involves $Qd(p)$. Filnally, $\tilde V$ is non-cyclic as
it has a non-cyclic homomorphic image $V$.

The other implication is clear.
\qed\end{proof}

\begin{lem}\label{constpcentric}
Let $Q$ be a $p$-subgroup and $P$ a Sylow $p$-subgroup of $G$ containing
a Sylow $p$-subgroup of $N_G(Q)$. Then any $p$-subgroup of $G$ that
contains $Q C_P(Q)$ is $p$-centric.
%(Compare with $\sim$ Lemma~4.42 in~\cite{craven}.)
\end{lem}

\begin{proof}
By construction, $C_P(Q)$ is a Sylow $p$-subgroup of $C_G(Q)$.
Let $c \in C_G(Q C_P(Q))$ be a $p$-element. Then $c$ centralises $Q$
and $C_P(Q)$, so $\langle c \rangle C_P(Q)$ is a $p$-group centralising
$Q$. Hence $c \in C_P(Q) \leq QC_P(Q)$ by the maximality of $C_P(Q)$.
\qed\end{proof}

\begin{prop}\label{QdpinNpcentric}
Let $G$ be a group that involves $Qd(p)$. Then $N_G(Q)$ involves $Qd(p)$
for a $p$-centric subgroup $Q$ of $G$.
\end{prop}

\begin{proof}
Let $K \triangleleft H \leq G$ such that $H/K = V \rtimes S \cong Qd(p)$.
By the proof of Lemma~\ref{plocal} we may assume $H \leq N_G(\tilde V)$
for a $p$-subgroup $\tilde V$ of $G$ and $W = K \cap \tilde V$ is a normal
subgroup of $H$.

As $S \cong SL_2(p)$, $S = \langle x, a \rangle$ for some $x$,
$a \in S$ such that $x^p = a^4 = 1$ and $[V, x, x] = 1$. Let moreover
$\tilde x$ and $\tilde a$ be preimages of $x$ and $a$ under the natural
homomorphism $H \to H/K$, respectively.

Let $Q$ be a Sylow $p$-subgroup of $\tilde V C_G(\tilde V / W)$. Then
$Q$ is $p$-centric by Lemma~\ref{constpcentric}. Let $H_1 = H C_G (\tilde
V/W)$ and $K_1 = K C_G(\tilde V/W) = C_G(\tilde V/W)$. The latter equality
holds because $[K, \tilde V] \subseteq K \cap \tilde V = W$. Observe that
$H_1$ is a subgroup of $G$ because $H$ normalises both $\tilde V$ and $W$.
Note that $V$ can be identified with $\tilde V/W$ and we do identify them.

Now, $K_1 = \tilde V K_1 \triangleleft H_1$ and $Q$ is a Sylow $p$-subgroup
of $K_1$. Hence by Frattini argument we have
\[
H_1 = 
%N_{H_1}(Q) \cdot \tilde V K_1 = 
N_{H_1}(Q) \cdot K_1.
\]
Then $\tilde x = n_x \cdot k_x$ and $\tilde a = n_a \cdot k_a$ for appropriate
elements $n_x$, $n_a \in N_{H_1}(Q)$ and $k_x$, $k_a \in K_1$.

Consider the factor group $\bar N = N_{H_1}(Q) / W$. By construction,
$V = \tilde V / W \triangleleft \bar N$. Let $\bar x$ and $\bar a$ be
the images under the natural homomorphism $N_{H_1}(Q) \to \bar N$,
of $n_x$ and $n_a$, respectively. Then the operations of $x$ and $\bar x$
on $V$ coincide, just as those of $a$ and $\bar a$, because $K_1$
centralises $V$.

Therefore, $[V, \bar x, \bar x] = 1$, where $\bar x \in N_{\bar N}(V) =
\bar N$.
%Moreover, $[V, \bar x^{\bar a}, \bar x^{\bar a}] = 1$, too. Now, 
The image of $\langle \bar x, \bar a \rangle$ in $\bar N / C_{\bar N}(V)$
is isomorphic to $SL_2(p)$ and hence
\[
\bar x \notin O_p(\bar N / C_{\bar N}(V)).
\]
This means that $\bar N$ is not $p$-stable, so it involves $Qd(p)$ by
Glauberman's Theorem~\ref{glaubthm}. It follows that $N_{H_1}(Q)$ and
hence $N_G(Q)$ involve $Qd(p)$.
\qed\end{proof}

Now we are ready to prove the theorem.

\begin{proof}[of Theorem~$\ref{freeinv}$]
Assume $G$ involves $Qd(p)$. Then $Qd(p)$ is involved in $N_G(Q)$ for some
$p$-centric subgroup $Q$ of $P$ by Proposition~\ref{QdpinNpcentric}.
Observe that some conjugate of $Q$ is fully $\mc F$-normalised and also
$\mc F$-centric (the latter by Lemma~\ref{F=pcentric}). Since $Qd(p)$ has
no normal $p'$-subgroups, it is also involved in $N_G(Q) / O_{p'}(N_G(Q))$.
As this group is the model of $N_{\mc F}(Q)$ by Lemma~\ref{modelofN},
$\mc F$ is not $Qd(P)$-free.

For the converse, assume $\mc F$ is not $Qd(p)$-free. Then $Qd(p)$ is
involved in $N_G(Q)/O_{p'}(N_G(Q))$ for some $\mc F$-centric subgroup $Q$
of $p$ by definition. Therefore, $Qd(p)$ is also involved in $G$.
\qed\end{proof}

The following corollary is a slight refinement of Glauberman's
Theorem~\ref{glaubthm}

\begin{cor} \label{refglaubthm}
The following are equivalent:
\begin{itemize}
\item All sections of $G$ are $p$-stable.
\item $N_G(Q)$ does not involve $Qd(p)$ for any $p$-centric $p$-subgroup
$Q$ of $G$.
\end{itemize}
\end{cor}

\section{Section {\itshape p}-stability in fusion systems}
\label{secsectpstab}
%The corresponding notion for fusion systems is as follows:
%
%\begin{Def} \label{sectionpstabfussys}
%Let $\mc F$ be a fusion system on the $p$-group $P$. We call $\mc F$
%{\itshape section $p$-stable} if for all fully $\mc F$-normalised and
%$\mc F$-centric subgroups $Q$ of $P$ the model of $\mc N_{\mc F} (Q)$,
%the normaliser of $Q$ in $\mc F$, is a section $p$-stable group.
%\end{Def}
%
%By Corollary~\ref{refglaubthm} it is straightforward that a fusion
%system $\mc F$ is section $p$-stable if and only if it is $Qd(p)$-free.
%
%\begin{rem}
%It follows from~\cite{kessar:linckelmann:2008} that a $Qd(p)$-free
%fusion system $\mc F$ is soluble. Indeed, Theorem~B asserts that $Z(J(P))$
%is normal in $\mc F$. Now, by Proposition~6.4, $\mc F/Z(J(P))$ is also
%$Qd(p)$-free. Since $Z(J(P))$ is non-trivial, the claim follows by
%induction.
%
%Being soluble, $\mc F$ is constrained and hence it has a model.
%By Theorem~\ref{freeinv} it follows that not only the model of a $\mc F$,
%but any group whose fusion system is $\mc F$ has the property that it does
%not involve $Qd(P)$.
%\end{rem}
%
%The question is now, whether there are exotic fusion systems that are
%$p$-stable.

%\begin{prop}
%Let $\mc F$ be a section $p$-stable fusion system on the $p$-group $P$.
%Then for all subsystems $\mc G$ of $\mc F$ (on some subgroup $R$ of $P$)
%and for all subgroups $Q$ of $R$, the fusion system $\mc G/Q$ is section
%$p$-stable.
%\end{prop}
%
%\begin{proof}
%As $\mc F$ is soluble, so are all subsystems $\mc G$ and all factors
%$\mc G/Q$. 
%\qed\end{proof}
We have seen in the case of groups that $p$-stability in itself is not
enough: one needs the notion of section $p$-stability. Two possible
definitions seem to be natural:

\begin{Def}\label{ver1sectpstabfussys}
Let $\mc F$ be a fusion system on the $p$-group $P$. $\mc F$ is called
{\itshape section $p$-stable} if $\mc N_{\mc F}(R) / R$ is $p$-stable
for all fully $\mc F$-normalised subgroups $R$ of $P$.
\end{Def} 

\begin{Def}\label{ver2sectpstabfussys}
Let $\mc F$ be a fusion system on the $p$-group $P$. $\mc F$ is called
{\itshape section $p$-stable} if the model of $\mc N_{\mc F}(R)$ is
section $p$-stable for all $\mc F$-centric and fully $\mc F$-normalised
subgroups $R$ of $P$.
\end{Def}

Clearly, Definition~\ref{ver2sectpstabfussys} is equivalent to
Definition~\ref{Qdpfree} of a $Qd(p)$-free fusion system.

We show that Definitions~\ref{ver1sectpstabfussys}
and~\ref{ver2sectpstabfussys} are equivalent.

% The following lemma is straightforward:
% 
% \begin{lem}\label{factfussys}
% Let $Q$ be a normal $p$-subgroup of $G$. Then
% \[
% \mc F_{P/Q} (G/Q) = \mc F_P(G) / Q.
% \]
% \end{lem}

\begin{theorem}\label{eqdefs}
A fusion system $\mc F$ is section $p$-stable according to
Definition~$\ref{ver1sectpstabfussys}$ if and only if it is section
$p$-stable according to Definition~$\ref{ver2sectpstabfussys}$.
\end{theorem}

\begin{proof}
Assume $\mc F$ is section $p$-stable according to
Definition~\ref{ver1sectpstabfussys}. Let $R$ be an $\mc F$-centric and
fully $\mc F$-normalised subgroup of $P$. Let $L$ be the model of
$\mc N_{\mc F}(R)$ with Sylow $p$-subgroup $S = N_P(R)$. We have to show
that $N_L(Q) / Q$ is $p$-stable for all subgroups of $S$. We can
assume $Q$ is fully $\mc N_{\mc F}(R)$-normalised. Then a Sylow $p$-subgroup
of $N_L(Q)$ is $N_S(Q)$ and the corresponding fusion system is
\[
\mc F_{N_S(Q)}(N_L(Q)) = \mc N_{\mc N_{\mc F}(R)} (Q).
\]
Let $\mc N = \mc N_{\mc N_\mc F(R)} (Q)$. By Theorem~5.20
in~\cite[p. 145]{craven}, we have
\[
\mc F_{N_S(Q)/Q}\big(N_L(Q)/Q\big) = \mc N / Q
\]
follows. In view of Theorem~\ref{stabeqpstab} we have to show that
$\mc N / Q$ is $p$-stable.

Let $Q_1$ be a fully $\mc F$-normalised member of the
$\mc F$-isomorphism class of $Q$. Then there is an $\mc F$-morphism
$\phi$: $N_P(Q) \to N_P(Q_1)$ extending an isomorphism $Q \to Q_1$
(see e. g. Lemma~2.2 in~\cite{kessar:linckelmann:2008}).
Then by Lemma~\ref{normfunctor}, $\phi$ induces an injective functor
\[
\Phi:\ \mc N\to \mc N_{\mc F}(Q_1)
\]
and hence $\mc N$ can be identified with a subsystem of $\mc N_{\mc F}(Q_1)$.

We now claim that $\Phi$ induces an injective functor
\[
\overline\Phi:\ \mc N / Q \to \mc N_{\mc F}(Q_1) / Q_1.
\]
Indeed, for all objects $T \geq Q$ of $\mc N$ we have $\Phi(T) = T \phi
\supseteq Q \phi = Q_1$, so we may define $\overline \Phi(T/Q) =
T\phi / Q_1$. Let $\psi$: $T \to S$ be a morphism in $\mc N$ which
induces the morphism $\bar\psi$: $T/Q \to S/Q$ of $\mc N/Q$. Then $\psi^\phi$
induces a morphism $\overline{\psi^\phi}$ in $\mc N_{\mc F}(Q_1) / Q_1$.
What we have to show is the following: $\bar\psi_1 = \bar\psi_2$ if and
only if $\overline{\psi_1^\phi} = \overline{\psi_2^\phi}$. In other words,
$t\psi_1 Q = t \psi_2 Q$ for all $t \in T$ if and only if
$(t\phi) \psi_1^\phi Q_1 = (t\phi) \psi_2^\phi Q_1$ for all $t \in T$.
But this is clear by the definition of $\psi_1^\phi$ and $\psi_2^\phi$.
% 
% To show that $\overline\Phi$ is injective, let $\psi_1$ and $\psi_2$:
% $T \to S$ be morphism of $\mc N$ such that $\overline{\psi_1^\phi} =
% \overline{\psi_2^\phi}$. We have to show that then  $\bar\psi_1 =
% \bar\psi_2$

Identified with a subsystem of the $p$-stable fusion system
$\mc N_{\mc F}(Q_1) / Q_1$, the system $\mc N_{\mc N_{\mc F}(R)} (Q) / Q$
is $p$-stable. Hence $\mc F$ is section $p$-stable according to
Definition~\ref{ver2sectpstabfussys}.

Assume now that $\mc F$ is section $p$-stable according to
Definition~\ref{ver2sectpstabfussys}. Then $\mc F$ is $Qd(p)$-free and
hence constrained by Remark~\ref{Qdpfreesol}. Its model $G$ is $Qd(p)$-free,
therefore section $p$-stable by Theorem~\ref{freeinv}. Now,
$\mc N_{\mc F}(Q)/Q$ is the fusion system of $N_G(Q)/Q$ for all fully
$\mc F$-normalised subgroups $Q$ of $P$. As $N_G(Q)/Q$ is $p$-stable, so is
$\mc N_{\mc F}(Q)/Q$.
\qed\end{proof}

\begin{prop}
The fusion system $\mc F$ is section $p$-stable if and only if for all
subsystems $\mc G$ of $\mc F$ and all subgroups $Q$ of $P$ such that
$Q \triangleleft \mc G$ the quotient system $\mc G / Q$ is $p$-stable.
\end{prop}

\begin{proof}
If all subquotients are $p$-stable, then so are the fusion systems
$\mc N_\mc F(R) / R$ for all fully $\mc F$-normalised subgroups $R$ of $P$.
Hence we only have to prove the other implication.

Let $\mc F$ be section $p$-stable and let $\mc G$ be an arbitrary subsystem
of $\mc F$ with $Q \triangleleft \mc G$. Let $Q_1$ be a fully
$\mc F$-normalised subgroup of $P$ that is $\mc F$-isomorphic to $Q$. By
the same line of arguments as in Theorem~\ref{eqdefs}, $\mc G / Q$ is
isomorphic to a subsystem of $\mc N_\mc F(Q_1) / Q_1$ and, as such, it is
$p$-stable.
\end{proof}

\section{On fusion systems on extraspecial {\itshape p}-groups of order
\texorpdfstring{{\itshape p}\textsuperscript{3}}{{\itshape p} cubed} and exponent {\itshape p}}
\label{extraspecfussysts}
Let $P$ be an extraspecial group of order $p^3$ and exponent $p$.
All fusion systems over $P$ were classified by A.~Ruíz and A.~Viruel in
\cite{ruiz:viruel:2004}. In this section we determine which of these
fusion systems are $p$-stable. This might be crucial in the study of
$p$-stability since this group is the Sylow $p$-subgroup of $Qd(p)$.

We examine the following questions:
%\begin{itemize}
%\item 
Which of these fusion systems are $p$-stable?
%\item 
Which of these fusion systems are section $p$-stable (equivalently,
$Qd(p)$-free)?
%\item 
Which of these fusion systems are soluble?
%\end{itemize}

By Alperin's fusion theorem, a fusion system is completely determined by
the groups $\Aut_\mc F (P)$ and $\Aut_\mc F (R)$, where $R$ ranges over the
set of essential subgroups of $P$. Our first observation is that essential
subgroups of $P$ in our case are precisely the radical subgroups and they
are elementary Abelian of order $p^2$. By this, $\mc F$ is $p$-stable if
and only if $SL_2(p)$ is not contained in $\Aut_\mc F (R)$ for any radical
subgroup $R$ of $p$. Having a look at the tables describing the fusion
systems on $P$ (see Tables~9.1 and~9.2 in \cite[pp. 321, 323]{craven}), we
obtain the result:

\begin{prop}
Let $P$ be an extraspecial group of order $p^3$ and exponent $p$. Then
all fusion systems defined on $P$ are non-$p$-stable except for the
fusion system of $G = P \rtimes H$ ($p \ndiv |H|$), which is section
$p$-stable.
\end{prop}

Concerning solubility, we can establish that $\mc F$ is soluble if and
only if $P$ has a non-trivial strongly closed Abelian subgroup.
By Proposition~4.61 in~\cite[p. 129]{craven} applied to this case, $Q$ is
normal in $\mc F$ if and only if it is contained in every radical subgroup
of $P$.

Therefore, if $P$ has at least two radical subgroups, then the only
possibility for an $\mc F$-normal subgroup is $Z(P)$. However, $SL_2(p)$ is
contained in $\Aut_\mc F(R)$ for all fusion systems with at least two
radical subgroups. Hence $Z(P)$ is not fixed under the action of
$\Aut_\mc F(R)$, so $(Z(P) \ntriangleleft \mc F$ in this case.

If $P$ has exactly one radical subgroup $R$, then certainly $R
\triangleleft \mc F$, so $\mc F$ is soluble in this case. Since the group
$P \rtimes H$ with $p \ndiv |H|$ is $p$-soluble (in which case there are
no radical subgroups), its fusion system is trivially soluble.

Summarising this, we obtain:

\begin{prop}
Let $P$ be an extraspecial group of order $p^3$ and exponent $p$ and let
$\mc F$ be a fusion system on $P$. Then $\mc F$ is soluble if and only if
$P$ has at most one radical subgroup, that is, if $G \cong P \rtimes H$
with $p \ndiv |H|$ or $G \cong R \rtimes (SL_2(p) \rtimes C_r)$ with
$r | p-1$.
\end{prop}

\section{Concluding remarks and questions}\label{finalrems}
In Sections~\ref{Qdsecdefchar} to \ref{sporadic} we have shown that
a finite simple group is $p$-stable if and only if it is section
$p$-stable. Moreover, we have proved that a non-$p$-stable simple group
contains a subgroup isomorphic to either $Qd(p)$ or $\widetilde{Q}(p)$,
or, if $p = 3$, $\widetilde{Qd}^-(3)$ or $3^2{:}(2^2{^.}SL_2(3))$. Also,
we determined the complete list of finite simple groups with this property
by showing that one of the above groups are contained in them.
We emphasise, however, that our list is not complete in the sense that a
finite simple group may contain more than one group from the above list
even if it has not been proven here. Also, it may contain a minimal
non-$p$-stable group not listed here.
%we are far from completely describing all sections
%of finite simple groups isomorphic to $Qd(p)$.

By all these, the question naturally arises: which groups are minimal
non-$p$-stable at all? By the results presented here, these groups have a
factor group isomorphic to $Qd(p)$, but this is not a sufficient condition:
in Example~\ref{pstabnonsectpstab_ex}, we have found a $p$-stable group with
$Qd(p)$ as a factor group. It might be a reachable project to determine all
minimal non-$p$-stable groups that occur as subgroups of finite simple
groups.
%In the present paper we have found the groups listed in the
%previous paragraph.

By an old result, if a group is soluble, then it is section $p$-stable, but
section $p$-stability does not imply solubility. For fusion systems, the
converse is true: if a fusion system is section $p$-stable, then it is
soluble, but a soluble fusion system need not be section $p$-stable (as for
the fusion system of $Qd(p)$ itself).

Also, for fusion systems of finite simple groups we have seen that
$p$-stability and section $p$-stability are equivalent notions.
%Again, the question arises whether this is a general phenomenon or there are
However, this is not a general phenomenon as the fusion system of the group
in Example~\ref{pstabnonsectpstab_ex} is $p$-stable but not section
$p$-stable. Nevertheless, all of our examples of $p$-stable fusion systems
are soluble as well. So the question arises: Are there $p$-stable fusion
systems that are not soluble?

As soluble fusion systems have models, we can also ask: Are there exotic
$p$-stable fusion systems? Recall that in Section~\ref{extraspecfussysts},
the exotic ones were all non-$p$-stable, so we do not have any examples for
that at the moment.

%As a weakening of our question on solubility, we may ask: Have all
%$p$-stable fusion system a non-trivial $O_p$?

\bibliographystyle{alpha}
\bibliography{pstable}

\end{document}